\documentclass[11pt]{article}

\usepackage[a4paper,margin=27mm]{geometry}
\usepackage{amsmath,amssymb,amsthm,mathtools}

\usepackage{booktabs,array,multirow}
\usepackage{tikz}
\usepackage{array,longtable}
\usetikzlibrary{matrix,arrows,positioning,fit,patterns}
\usepackage[dvipsnames]{xcolor}

\usepackage[colorlinks=true,linkcolor=MidnightBlue,citecolor=OliveGreen,urlcolor=BrickRed]{hyperref}
\usepackage{microtype}
\hypersetup{
  pdftitle={Equidistribution of Mesh Patterns of Length Two on Involutions},
  pdfauthor={Author name omitted for review}
}

\newtheorem{theorem}{Theorem}[section]
\newtheorem{lemma}[theorem]{Lemma}

\theoremstyle{definition}

\newtheorem{remark}[theorem]{Remark}
\newtheorem{example}[theorem]{Example}

\newcommand{\Sn}{\mathfrak S_n}
\newcommand{\In}{\mathcal I_n}
\newcommand{\occ}{\operatorname{occ}}
\newcommand{\pos}{\operatorname{pos}}
\newcommand{\Occ}{\operatorname{Occ}}
\newcommand{\cl}{\operatorname{cl}}
\newcommand{\Des}{\operatorname{Des}}
\newcommand{\lmin}{\operatorname{lmin}}
\newcommand{\rmin}{\operatorname{rmin}}
\newcommand{\rmax}{\operatorname{rmax}}

\newcommand{\class}[1]{\textnormal{INV-#1}}
\newcommand{\pattern}[4]{%
  \raisebox{0.6ex}{%
    \begin{tikzpicture}[scale=0.35,baseline=(current bounding box.center),#1]
      \foreach \x/\y in {#4} \fill[gray!50] (\x,\y) rectangle +(1,1);
      \draw (0.01,0.01) grid (#2+0.99,#2+0.99);
      \foreach \x/\y in {#3} \filldraw (\x,\y) circle (6pt);
    \end{tikzpicture}}%
}
\newcommand{\pspace}{\hspace{0.7mm}\allowbreak}
\newcolumntype{C}[1]{>{\centering\arraybackslash}p{#1}}
\newcommand{\orbitcell}[1]{\rule[-0.4mm]{0pt}{5.8mm}#1}

\begin{document}
\begin{center}
{\large \bf  Equidistribution of mesh patterns of length 2 on involutions}
\end{center}

\begin{center}
Xinyu Su$^{a}$ and Philip B. Zhang$^{b*}$
\\[6pt]

$^{a,b}$College of Mathematical Sciences \& Institute of Mathematics and Interdisciplinary Sciences, Tianjin Normal University, \\ Tianjin  300387, P. R. China\\[6pt]
Email:
 $^{a}${\tt suxinyu2220@163.com},
      $^{b*}${\tt zhang@tjnu.edu.cn}
\end{center}

\begin{abstract}
We classify all $1024$ length-two mesh patterns on involutions into
exactly $287$ equidistribution classes, consisting of
$262$ trivial classes and $25$ nontrivial classes. Since many bijections on
arbitrary permutations do not preserve the involution condition, our
proofs use several involution-preserving bijections, including
conjugation and local adjacent-transposition operations, transposition of the RSK recording tableau,
prefix-rank swap bijections and marked-occurrence bijections. In particular, we establish the
equidistribution of $\class{099}$, a nontrivial class closely related
to the unrestricted Class~69, by adapting the
marked-occurrence bijection of Zhang and Zhao through a symmetric
closure of the marked set. We also obtain explicit distribution formulas and recurrences for
several classes, including continued fractions obtained by
specializing known labeled Motzkin-path formulas and rational
expressions derived from paired skew decompositions of involutions.
\end{abstract}

\vspace{1em}


\section{Introduction}
\label{sec}
Let $\mathfrak{S}_n$ denote the symmetric group on $[n]:=\{1,2,\ldots,n\}$. Mesh patterns, introduced by Br\"and\'en and Claesson~\cite{BrClaesson}, generalize classical permutation patterns. A mesh pattern is a pair $(\tau,R)$, where $\tau\in\mathfrak S_k$ and
$R\subseteq \{0,1,\ldots,k\}\times\{0,1,\ldots,k\}$ specifies a set of shaded cells.
 An occurrence of $(\tau,R)$ in $\pi\in\mathfrak{S}_n$ is an occurrence of $\tau$ for which the corresponding shaded regions of the induced grid contain no additional points of the plot of $\pi$. For example, the following length-three mesh pattern
$$
\pattern{scale=0.6}{3}{1/1,2/3,3/2}{0/0,1/2,2/1,2/3,3/0,3/1}
$$
has underlying pattern $132$, and its shaded cells indicate the regions that must contain no additional points. The classification of the equidistribution of length-two mesh patterns on unrestricted permutations has been studied extensively in \cite{Hilmarsson,KitaevZhang,HanZeng,SuKitaevZhang,FangEtAl} and was completed by Zhang and Zhao~\cite{ZhangZhao}: there are exactly 105 equidistribution classes and 46 Wilf-equivalence classes on $\mathfrak{S}_n$.

We now restrict the ambient set from $\mathfrak S_n$ to $\mathcal I_n$, the set of involutions in $\mathfrak S_n$, and study the corresponding distributions of mesh pattern occurrences. This restriction changes the available proof methods.
Several bijections used on $\mathfrak{S}_n$, such as complementing values on one side of the maximum, reversing a prefix, or moving blocks, need not preserve the condition $\pi=\pi^{-1}$. At the same time, the plot of every involution is symmetric about the main diagonal, and this symmetry enables arguments specific to the involution setting.

Pattern statistics on involutions have been studied in several related settings. Barnabei, Castronuovo, and Silimbani~\cite{BarnabeiCastronuovoSilimbani} obtained a general formula for enumerating occurrences of Hertzsprung patterns in involutions and determined the distributions of all Hertzsprung patterns of lengths two and three. An occurrence of a Hertzsprung pattern is contiguous in both positions and values. Thus, Hertzsprung patterns form a special subclass of bivincular patterns, and hence of mesh patterns, so these results do not directly extend to the general mesh conditions considered here. Fang, Fu, Kitaev, Li, Su, and Sun~\cite{FangEtAl}
also proposed related equidistribution conjectures for short mesh
patterns on involutions, which motivate the present study. The Eulerian distribution on involutions has also been studied extensively; see Dukes~\cite{Dukes}, Guo and Zeng~\cite{GuoZeng}, Chow~\cite{Chow}, and Wang~\cite{Wang}.

Our main result gives a complete classification of length-two mesh patterns up to equidistribution on involutions. The inverse and reverse-complement operations partition the patterns into symmetry orbits. We call an equidistribution class \emph{trivial} if it consists of a single symmetry orbit, and \emph{nontrivial} otherwise. The $1024$ patterns form precisely $287$ equidistribution classes, of which $262$ are trivial and $25$ are nontrivial. The nontrivial classes require additional enumerative, geometric, and bijective arguments.

Besides completing the classification, we obtain explicit distribution
formulas for a number of the nontrivial classes. For $\class{033}$ and
$\class{176}$, we identify the relevant occurrences with
statistics in the labeled Motzkin-path model of Barnabei et
al.~\cite{BarnabeiMaxima}. The corresponding continued fractions then follow from their Theorem~3.5 by suitable specializations, within the classical continued-fraction framework of Flajolet~\cite{Flajolet}.
The distribution for $\class{161}$ is obtained from that of
$\class{033}$ by a reciprocity relation. Other classes lead naturally to recurrences obtained by examining the cycle containing the largest element. For example, the distributions in INV-004 and those arising from left-to-right minima admit explicit recurrences.

A second recurring method is decomposition into indecomposable
components. In the canonical skew decomposition of an involution,
inversion reverses the components, so the noncentral components occur
in inverse pairs around a possibly empty central component. By
assigning different weights to the outer pair, the inner pairs, and
the center, we obtain the distribution formulas for
$\class{041}$ and $\class{272}$. Direct-sum decompositions are used
similarly in several other cases; for instance, the distribution in
$\class{088}$ is governed by singleton direct-sum components, while
the classes $\class{218}$, $\class{219}$, and $\class{228}$ are
enumerated by identifying the direct-sum components that support an
occurrence.

Several further classes admit direct but nontrivial enumerations.
For $\class{085}$ and $\class{250}$, involutivity turns the shading
conditions into systems of forced transpositions and yields explicit
finite sums involving the involution numbers. For $\class{125}$, the first block determined by the leftmost occurrence leads to a functional equation and an explicit generating function, whereas the enumeration for
$\class{102}$ involves an inclusion--exclusion argument. These
examples illustrate that the involution condition, although it
restricts the bijections available on arbitrary permutations, often
makes the geometry of a mesh occurrence sufficiently rigid to permit
exact enumeration.

Our bijective arguments exploit this rigidity in a different way.
Conjugation and local adjacent-transposition operations preserve involutions while modifying selected regions of the permutation. The RSK
correspondence, followed by transposition of the recording tableau,
gives an involution-preserving bijection that maps the descent set to its complement and  applies to $\class{046}$ and $\class{052}$. Prefix-rank and local swapping bijections are compatible with inversion and act directly on $\mathcal I_n$.

One of the most delicate nontrivial cases is INV-099, which is closely related to Class~69 of the unrestricted classification~\cite{SuKitaevZhang}. Two of its representatives belong to Class~69, while the other two are complements of the remaining representatives of that class. Although complementation defines a natural bijection on $\mathfrak S_n$, it does not preserve the set of involutions.
We therefore adapt the marked-occurrence
bijection of Zhang and Zhao~\cite{ZhangZhao}. The marked occurrences
form star and path forests, and closing the marked set under reflection
in the main diagonal makes the componentwise star--path transformation
compatible with the symmetry of an involution plot. This yields an
involution-preserving marked-occurrence bijection and completes the
classification.

The rest of the paper is organized as follows. Section~\ref{sec:definitions} introduces the required notation, defines equidistribution, and records the inverse and reverse-complement symmetries. Section~\ref{sec:enumeration} proves thirteen nontrivial classes through pointwise occurrence descriptions and direct enumeration. Section~\ref{sec:involution-plot-symmetry} uses reflection across the main diagonal to prove four further classes. Section~\ref{sec:involution-preserving-bijections} constructs involution-preserving bijections for seven classes. Section~\ref{sec:marked-occurrence-bijection} proves the remaining class \class{099} using a marked-occurrence bijection. Finally, Section~\ref{conclude} summarizes the results and discusses possible extensions, while the \hyperref[app:trivial-classes]{Appendix} lists the trivial equidistribution classes.

\begin{table}[!t]
\centering
\begingroup

\renewcommand{\orbitcell}[1]{\rule[-2.4mm]{0pt}{6.8mm}#1}
\small
\setlength{\tabcolsep}{0.2mm}
\renewcommand{\arraystretch}{1.10}
\setlength{\extrarowheight}{0.8mm}
\begin{tabular}{|>{\centering\arraybackslash}m{10.5mm}|>{\centering\arraybackslash}m{24.5mm}|>{\centering\arraybackslash}m{24.5mm}|>{\centering\arraybackslash}m{15mm}||>{\centering\arraybackslash}m{10.5mm}|>{\centering\arraybackslash}m{24.5mm}|>{\centering\arraybackslash}m{24.5mm}|>{\centering\arraybackslash}m{15mm}|}
\hline
\textbf{nr.} & \textbf{trivial} & \textbf{trivial} & \textbf{ref.} &
\textbf{nr.} & \textbf{trivial} & \textbf{trivial} & \textbf{ref.}\tabularnewline \hline
% left: INV-004 orbits 1-2/2; right: INV-114 orbits 1-2/11
{\scriptsize\texttt{INV-004}} & \orbitcell{\pattern{scale=0.5}{2}{1/1,2/2}{0/0,1/0}\hspace{0.8mm}\pattern{scale=0.5}{2}{1/1,2/2}{0/0,0/1}\hspace{0.8mm}\pattern{scale=0.5}{2}{1/1,2/2}{2/1,2/2}\hspace{0.8mm}\pattern{scale=0.5}{2}{1/1,2/2}{1/2,2/2}} & \orbitcell{\pattern{scale=0.5}{2}{1/1,2/2}{1/0,1/1}\hspace{0.8mm}\pattern{scale=0.5}{2}{1/1,2/2}{0/1,1/1}\hspace{0.8mm}\pattern{scale=0.5}{2}{1/1,2/2}{1/1,2/1}\hspace{0.8mm}\pattern{scale=0.5}{2}{1/1,2/2}{1/1,1/2}} & {\centering\small \shortstack{Thm~\ref{thm:004}}} & \multirow[c]{6}{*}{\scriptsize\texttt{INV-114}} & \orbitcell{\pattern{scale=0.5}{2}{1/1,2/2}{0/0,1/0,2/0,0/1,1/1,2/1,0/2,1/2}\hspace{0.8mm}\pattern{scale=0.5}{2}{1/1,2/2}{1/0,2/0,0/1,1/1,2/1,0/2,1/2,2/2}} & \orbitcell{\pattern{scale=0.5}{2}{1/1,2/2}{0/0,1/0,2/0,0/1,1/1,2/1,2/2}\hspace{0.8mm}\pattern{scale=0.5}{2}{1/1,2/2}{0/0,1/0,2/0,1/1,2/1,1/2,2/2}\hspace{0.8mm}\pattern{scale=0.5}{2}{1/1,2/2}{0/0,1/0,0/1,1/1,0/2,1/2,2/2}\hspace{0.8mm}\pattern{scale=0.5}{2}{1/1,2/2}{0/0,0/1,1/1,2/1,0/2,1/2,2/2}} & \multirow[c]{6}{=}{\centering\small \shortstack{Thm~\ref{thm:114}}} \\
\end{tabular}
\endgroup
\caption{All patterns in the twenty-five nontrivial involution classes. Each pattern column contains one  $\langle i,rc\rangle$-orbit.}
\label{tab:pattern-guide}
\end{table}

\section{Preliminaries}
\label{sec:definitions}

Write $[n]=\{1,\ldots,n\}$, and let $\Sn$ be the symmetric group on
$[n]$. Permutations are written in one-line notation, so
$\pi=\pi_1\cdots\pi_n$, where $\pi_i=\pi(i)$. Permutation products are
composed from right to left. For $g,\pi\in\Sn$, the permutation
$g\pi g^{-1}$ is the conjugate of $\pi$ by $g$; when $g^2=\mathrm{id}$,
this is simply $g\pi g$.

Let
$$\In=\{\pi\in\Sn:\pi=\pi^{-1}\},
\qquad
I_n=|\In|
=\sum_{k=0}^{\lfloor n/2\rfloor}
\frac{n!}{2^k k!(n-2k)!},$$
where the latter formula is classical; see, for example, \cite{StanleyEC2}. Thus $\In$ is the set of involutions in $\Sn$, whereas $I_n$ is its
cardinality. For $a\in[n]$, let
$\pos_\pi(a)=\pi^{-1}(a)$
denote the position of the value $a$ in $\pi$. In particular, if
$\pi\in\In$, then
$\pos_\pi(a)=\pi_a.$

We use the ordinary generating functions
$$
  \mathcal P(x)=\sum_{n\geq0}n!x^n,
  \qquad
\mathcal I(x)=\sum_{n\geq0}I_nx^n.
$$

The descent set and descent number of $\pi$ are
\[
\Des(\pi)=\{i\in[n-1]:\pi_i>\pi_{i+1}\},
\qquad
\operatorname{des}(\pi)=|\Des(\pi)|,
\]
and
$\operatorname{asc}(\pi)=n-1-\operatorname{des}(\pi),$ for $n\ge 1$. For $n=0$, both statistics are zero.
We write
\begin{align*}
\lmin(\pi)
&=\#\{j:\pi_j<\pi_i\text{ for every }i<j\},\\
\rmin(\pi)
&=\#\{j:\pi_j<\pi_i\text{ for every }i>j\},\\
\rmax(\pi)
&=\#\{j:\pi_j>\pi_i\text{ for every }i>j\}
\end{align*}
for the numbers of left-to-right minima, right-to-left minima, and
right-to-left maxima of $\pi$, respectively. The corresponding elements
are called left-to-right minima, right-to-left minima, and
right-to-left maxima. Left-to-right maxima are defined analogously.
For a statement $A$, let $\mathbf 1_{\{A\}}$ be $1$ when $A$ holds and
$0$ otherwise.

For
$\pi\in\Sn$, let $\Occ_p(\pi)$ be the set of occurrences of $p$ in
$\pi$, and put $\occ_p(\pi)=|\Occ_p(\pi)|$.
Define the bivariate generating function for the distribution of $\occ_p$ over involutions by
$$
F^{\mathcal I}_p(x,q)
=
\sum_{n\geq0}x^n
\sum_{\pi\in\mathcal I_n}q^{\occ_p(\pi)}.
$$
Two patterns $p$ and $p'$ are \emph{equidistributed} on involutions,
written
$p\sim_{\mathcal I}p',$
if
$F^{\mathcal I}_p(x,q)=F^{\mathcal I}_{p'}(x,q).$
Once a class $\class{c}$ has been proved to be
equidistributed, its common series is denoted by
$F^{\mathcal I}_{\class{c}}(x,q).$

\begin{lemma}
\label{lem:involution-symmetries}
Let
$  w_0=n(n-1)\cdots1,
  \pi^{rc}=w_0\pi w_0.$
For a mesh pattern $p=(\tau,R)$ of length $m$, define
\[
\begin{aligned}
  p^{-1}
  &=
  \bigl(\tau^{-1},R^{-1}\bigr),
  & R^{-1}
  &=
  \{(y,x):(x,y)\in R\},\\
  p^{rc}
  &=
  \bigl(\tau^{rc},R^{rc}\bigr),
  & R^{rc}
  &=
  \{(m-x,m-y):(x,y)\in R\},
\end{aligned}
\]
where $\tau^{rc}(i)=m+1-\tau(m+1-i).$
Also set
  $p^{rci}:=(p^{rc})^{-1}.$
Then, for every $n$,
\[
  p\sim_{\mathcal I} p^{-1}\sim_{\mathcal I} p^{rc}
  \sim_{\mathcal I} p^{rci}.
\]
\end{lemma}

\begin{proof}
Reflection in the main diagonal sends the plot of $\pi$ to that of
$\pi^{-1}$ and transforms every $p$-occurrence into a
$p^{-1}$-occurrence. Hence
  $\operatorname{occ}_{p}(\pi)
  =
  \operatorname{occ}_{p^{-1}}(\pi^{-1}).$
If $\pi\in\mathcal I_n$, then $\pi^{-1}=\pi$, so
  $\operatorname{occ}_{p}(\pi)
  =
  \operatorname{occ}_{p^{-1}}(\pi).$
Thus $p\sim_{\mathcal I}p^{-1}$.

Next, reverse-complement preserves involutions, since
\[
\begin{aligned}
  (\pi^{rc})^{-1}
  &=(w_0\pi w_0)^{-1}
  &=w_0\pi^{-1}w_0
  &=w_0\pi w_0
   =\pi^{rc}.
\end{aligned}
\]
Moreover, $\pi\mapsto\pi^{rc}$ is a bijection of $\mathcal I_n$.
Geometrically, it rotates the permutation plot through $180^\circ$,
and therefore
  $\operatorname{occ}_{p}(\pi)
  =
  \operatorname{occ}_{p^{rc}}(\pi^{rc}).$
Hence $p\sim_{\mathcal I}p^{rc}$. Finally, applying the inverse symmetry to
$p^{rc}$ gives
  $p^{rc}\sim_{\mathcal I}(p^{rc})^{-1}=p^{rci},$
which completes the proof.
\end{proof}

We call the orbit of a mesh pattern under the inverse and reverse-complement operations its \emph{symmetry orbit}. Thus the symmetry orbit of $p$ is
$
\{p,p^{-1},p^{rc},p^{rci}\},
$
with repetitions removed. By Lemma~\ref{lem:involution-symmetries}, every symmetry orbit is contained in a single equidistribution class. An equidistribution class is called \emph{trivial} if it consists of a single symmetry orbit, and \emph{nontrivial} otherwise.

For $\alpha\in\mathfrak S_r$ and $\beta\in\mathfrak S_s$, their
\emph{skew sum} $\alpha\ominus\beta$ is obtained by placing a copy
of $\alpha$ above and to the left of a copy of $\beta$. A nonempty
permutation is \emph{skew indecomposable} if it cannot be expressed
as the skew sum of two nonempty permutations. Every nonempty
permutation has a unique canonical skew decomposition into nonempty
skew indecomposable components.

\begin{lemma}
\label{lem:paired-skew-decomposition}
Every involution has a unique representation
$$
  \pi=
  \alpha_1\ominus\cdots\ominus\alpha_m
  \ominus\gamma
  \ominus\alpha_m^{-1}\ominus\cdots
  \ominus\alpha_1^{-1},
$$
where each $\alpha_i$ is a nonempty skew indecomposable permutation,
and $\gamma$ is either empty or a skew indecomposable involution. We
call $(\alpha_i,\alpha_i^{-1})$ an \emph{inverse pair} and $\gamma$
the \emph{center}.
\end{lemma}

\begin{proof}
Let $\pi=\beta_1\ominus\cdots\ominus\beta_k$ be the canonical skew
decomposition. Since inversion reverses a skew sum,
$\pi^{-1}=\beta_k^{-1}\ominus\cdots\ominus\beta_1^{-1}$. If
$\pi=\pi^{-1}$, uniqueness of the decomposition gives
$\beta_i=\beta_{k+1-i}^{-1}$ for every $i$. Pairing components
equidistant from the two ends proves the result; when $k$ is odd,
the middle component is itself an involution.
\end{proof}

Consequently, if $A(x)$ counts
one inverse pair and $C(x)$ counts the possible center, then the
corresponding involutions are generated by $C(x)/(1-A(x))$.

\begin{theorem}\label{thm:classification}
Length‑two mesh patterns fall into precisely $287$ equidistribution classes, of which $262$ are trivial and $25$ are nontrivial.
\end{theorem}

\begin{proof}
    To establish that there are exactly $287$ equidistribution classes, algebraic and bijective arguments show that the patterns within each class are equidistributed, and that these classes together
cover all $1024$ patterns. The distinctness of the classes is then verified by computing their distribution polynomials for $n\le 7$, which distinguish the $287$ representatives pairwise.
\end{proof}

For every distribution class considered below, we choose one
representative from each symmetry orbit contained in that
class. By Lemma~\ref{lem:involution-symmetries}, it is enough to prove that the displayed
representatives are equidistributed. We use this reduction
throughout without further mention.

\section{Equidistributions via enumeration}
\label{sec:enumeration}

The classes in this section are determined by explicit occurrence conditions.
Their distributions follow by counting the involutions satisfying those
conditions.

\begin{lemma}\label{lem:R-recurrence}
For $n\geq0$, define
$$
R_n(q)
=
\sum_{\pi\in\mathcal I_n}q^{\lmin(\pi)}
=
\sum_{\pi\in\mathcal I_n}q^{\rmax(\pi)},
\qquad
R(x,q)=\sum_{n\geq0}R_n(q)x^n.
$$
Then
$
R_0(q)=1, R_1(q)=q,
$
and
\begin{equation}\label{eq:R-recurrence}
R_{n+1}(q)
=
R_n(q)
+\bigl(n-1+q^2\bigr)R_{n-1}(q),
\qquad n\geq1.
\end{equation}
\end{lemma}

\begin{proof}
Since reverse-complement is a bijection on $\mathcal I_n$ and
interchanges left-to-right minima with right-to-left maxima,
$$
\sum_{\pi\in\mathcal I_n}q^{\lmin(\pi)}
=
\sum_{\pi\in\mathcal I_n}q^{\rmax(\pi)}.
$$

It remains to prove the recurrence. Consider the cycle containing
$n+1$ in an involution of $\mathcal I_{n+1}$. If $n+1$ is fixed,
deleting it preserves all left-to-right minima and contributes
$R_n(q)$.
Otherwise, suppose that $n+1$ is paired with $k\in[n]$. Deleting
the transposition $(k,n+1)$ and standardizing gives an involution
in $\mathcal I_{n-1}$. If $k=1$, the two inserted elements $n+1$
and $1$ are new left-to-right minima, while all existing
left-to-right minima are preserved. This case therefore contributes
$q^2R_{n-1}(q)$. For each of the remaining $n-1$ choices
$2\leq k\leq n$, no new left-to-right minimum is created, giving
$(n-1)R_{n-1}(q)$.

Combining these three cases yields Equation~\eqref{eq:R-recurrence}, as desired.
\end{proof}

\begin{theorem}\label{thm:004}
For every $n\geq0$, all patterns in $\class{004}$ are equidistributed
on $\mathcal I_n$. Moreover,
\[
F^{\mathcal I}_{\class{004}}(x,q)
=
R(qx,q^{-1}),
\]
where $R(x,q)$ is defined in Lemma~\ref{lem:R-recurrence}.
\end{theorem}

\begin{proof}
For the representatives
$$
p=\pattern{scale=0.6}{2}{1/1,2/2}{0/0,1/0}
\qquad\text{and}\qquad
p'=\pattern{scale=0.6}{2}{1/1,2/2}{1/0,1/1},
$$
let $a<b$ be the selected values of an occurrence. Every
element $b$ that is not a left-to-right minimum determines a unique
occurrence: for $p$, the shading forces $a$ to be the smallest
element to the left of $b$, while for $p'$ it forces $a$ to be the
rightmost element to the left of $b$ that is smaller than $b$.
Hence, for every $\pi\in\In$,
$\occ_p(\pi)=\occ_{p'}(\pi)=n-\lmin(\pi).$

Therefore, for $n\geq1$,
$$[x^n]F^{\mathcal I}_{\class{004}}(x,q)
=
\sum_{\pi\in\mathcal I_n}q^{\,n-\lmin(\pi)}
=
q^nR_n(q^{-1}).$$
The same identity holds for $n=0$, since $R_0(q)=1$. Summing over
$n\geq0$ gives
\[
F^{\mathcal I}_{\class{004}}(x,q)
=
\sum_{n\geq0}R_n(q^{-1})(qx)^n
=
R(qx,q^{-1}),
\]
as required.
\end{proof}

% For the recurrence, put
% $
% A_n(q)
% =
% [x^n]F^{\mathcal I}_{\class{004}}(x,q)
% =
% \sum_{\pi\in\mathcal I_n}q^{\,n-\lmin(\pi)}.
% $
% Clearly, $A_0(q)=A_1(q)=1$. Consider the cycle containing $n+1$.
% If $n+1$ is fixed, deleting it leaves the number of left-to-right
% minima unchanged and contributes $qA_n(q)$. If
% $n+1$ is paired with $k\in[n]$, deleting the transposition
% $(k,n+1)$ and standardizing gives an involution in
% $\mathcal I_{n-1}$. For $k=1$, the two inserted elements are new
% left-to-right minima, so the statistic $n-\lmin$ is unchanged; this
% contributes $A_{n-1}(q)$. For each of the $n-1$ choices
% $2\leq k\leq n$, no new left-to-right minimum is created, and the
% $n-\lmin$ increases by $2$, contributing
% $q^2A_{n-1}(q)$. This completes the proof of the recurrence.

% ===================
\begin{theorem}\label{thm:033}
For every $n\geq0$, all patterns in $\class{033}$ are equidistributed
on $\mathcal I_n$. Moreover,
\begin{equation}\label{eq:033-continued-fraction}
F^{\mathcal I}_{\class{033}}(x,q)
=
\frac{q-1}{q}
+
\cfrac{1/q}{
1-qx-\cfrac{qx^2}{
1-x-\cfrac{(q+1)x^2}{
1-x-\cfrac{(q+2)x^2}{
1-x-\ddots}}}}.
\end{equation}
\end{theorem}

% ================

\begin{proof}
For the representatives
$$
\begin{aligned}
p_1&=\pattern{scale=0.6}{2}{1/1,2/2}{0/0,1/0,2/0,2/1},
&
p_2&=\pattern{scale=0.6}{2}{1/1,2/2}{0/0,2/0,0/1,2/1},
&
p_3&=\pattern{scale=0.6}{2}{1/1,2/2}{1/0,2/0,1/1,2/1},
&
p_4&=\pattern{scale=0.6}{2}{1/1,2/2}{2/0,0/1,1/1,2/1},
\end{aligned}
$$
let $a<b$ be the selected values of an occurrence. For each
pattern, the shading forces $b$ to be a right-to-left minimum, and
necessarily $b\neq1$. Conversely, every right-to-left minimum
$b>1$ determines a unique occurrence: $p_1$ requires $a=1$;
$p_2$ and $p_3$ require $a$ to be the leftmost and rightmost element
smaller than $b$, respectively; and $p_4$ requires $a=b-1$.

Thus, for every $n\geq 1$ and $\pi\in\mathcal I_n$,
$
\occ_{p_1}(\pi)=\occ_{p_2}(\pi)=\occ_{p_3}(\pi)=\occ_{p_4}(\pi)=\rmin(\pi)-1.
$
Consequently,
\[
1+q\bigl(F^{\mathcal I}_{\class{033}}(x,q)-1\bigr)
=
1+\sum_{n\geq1}x^n
\sum_{\pi\in\In}q^{\rmin(\pi)}.
\]
By reverse-complement, the distribution of $\rmin$ on $\In$ is the
same as that of left-to-right maxima. Barnabei et
al.~\cite[Theorem~3.5]{BarnabeiMaxima} obtained the corresponding
continued fraction. Solving for
$F^{\mathcal I}_{\class{033}}(x,q)$ gives
\eqref{eq:033-continued-fraction}.
\end{proof}

% \begin{theorem}\label{thm:041}
% All patterns in $\class{041}$ are equidistributed.
% Moreover,
% \[
% [x^n]F^{\mathcal I}_{\class{041}}(x,q)
% =\sum_{\pi\in\In}
% q^{\,\#\{b:\,b\text{ is a right-to-left maximum having a
% smaller element to its left}\}}.
% \]
% \end{theorem}

% \begin{proof}
% For the representatives
% \[
% p=\pattern{scale=0.6}{2}{1/1,2/2}{0/0,1/0,2/2}
% \qquad\text{and}\qquad
% p'=\pattern{scale=0.6}{2}{1/1,2/2}{0/1,1/1,2/2},
% \]
% in any occurrence $a < b$, the shading restricts $b$ to be a right-to-left maximum having a smaller element to its left.
% Once such a $b$ is fixed, $p$ requires $a$ to be the smallest element to the left of $b$, while $p'$ requires $a$ to be the largest element to the left of $b$ that is smaller than $b$. 

% Since $b$ is guaranteed to have a smaller element to its left, $a$ is well-defined and unique in both cases. Thus, the number of occurrences for each pattern exactly equals the number of such elements $b$. The stated identities follow.
% \end{proof}

\begin{theorem}\label{thm:041}
For every $n\geq0$, all patterns in $\class{041}$ are equidistributed
on $\mathcal I_n$. Moreover,
\begin{equation}\label{eq:041-distribution}
F^{\mathcal I}_{\class{041}}(x,q)
=
\frac{R(x,q)+(1-q)xP(x,q)}
     {1-(1-q^2)x^2P(x,q)},
\end{equation}
where $R(x,q)$ is defined in Lemma~\ref{lem:R-recurrence} and
\[
P(x,q)
=
\sum_{m\geq0}(q^2)^{\overline m}x^{2m}.
\]
Here $(a)^{\overline m}=a(a+1)\cdots(a+m-1)$, with
$(a)^{\overline0}=1$.
\end{theorem}

\begin{proof}
For the representatives
\[
p=
\pattern{scale=0.6}{2}{1/1,2/2}{0/0,1/0,2/2}
\qquad\text{and}\qquad
p'=
\pattern{scale=0.6}{2}{1/1,2/2}{0/1,1/1,2/2},
\]
let $a<b$ be the selected values of an occurrence. The shading
forces $b$ to be a right-to-left maximum having a smaller element
to its left. Once $b$ is fixed, $p$ requires $a$ to be the smallest
element to the left of $b$, whereas $p'$ requires $a$ to be the
largest such element smaller than $b$. Hence $a$ is uniquely
determined in both cases.

Let $s_1(\pi)$ denote the number of singleton components in the
canonical skew decomposition of $\pi$. A right-to-left maximum
$b$ is not counted precisely when it has no smaller element to its
left. In this case, the elements to the left and right of $b$ are
$\{b+1,\ldots,n\}$ and $[b-1]$, respectively, so $b$ is a singleton
skew component. The converse is immediate. Thus, for every
$\pi\in\mathcal I_n$,
$\occ_p(\pi)=\occ_{p'}(\pi)=\rmax(\pi)-s_1(\pi).$

By Lemma~\ref{lem:paired-skew-decomposition}, every involution has
the form
$\alpha_1\ominus\cdots\ominus\alpha_m\ominus\gamma
\ominus\alpha_m^{-1}\ominus\cdots\ominus\alpha_1^{-1},$
where each $\alpha_i$ is nonempty and skew indecomposable, and
$\gamma$ is either empty or a skew indecomposable involution.
A point is a right-to-left maximum exactly when
its northeast region is empty. This property is
preserved by transposition, so
$\rmax(\alpha^{-1})=\rmax(\alpha)$.
Since
$\sum_{\alpha\in\mathfrak S_m}
q^{2\rmax(\alpha)}
=
(q^2)^{\overline m},$
the series $P(x,q)$ counts arbitrary inverse pairs
$(\alpha,\alpha^{-1})$ by total length and total number of
right-to-left maxima. Let $A(x,q)$ count such a pair with
$\alpha$ nonempty and skew indecomposable. Since every permutation
has a unique skew decomposition and $\rmax$ is additive under skew
sums,
$A(x,q)=1-{P(x,q)^{-1}}.$

Let $C(x,q)$ count the possible central component $\gamma$ by length
and $\rmax$. The paired skew decomposition gives
\[
R(x,q)
=
\frac{C(x,q)}{1-A(x,q)}
=
C(x,q)P(x,q),
\]
and hence
$C(x,q)=\frac{R(x,q)}{P(x,q)}.$

It remains to remove the $q$-weight of singleton components. A
singleton inverse pair has weight $q^2x^2$ in $A(x,q)$ and must
instead have weight $x^2$, while a singleton central component has
weight $qx$ in $C(x,q)$ and must instead have weight $x$. Summing over the number of inverse pairs gives
\[
\begin{aligned}
F^{\mathcal I}_{\class{041}}(x,q)
&=
\sum_{m\geq0}
\bigl(A(x,q)+(1-q^2)x^2\bigr)^m
\bigl(C(x,q)+(1-q)x\bigr)\\
&=
\frac{R(x,q)+(1-q)xP(x,q)}
     {1-(1-q^2)x^2P(x,q)}.
\end{aligned}
\]
This proves \eqref{eq:041-distribution}.
\end{proof}

\begin{theorem}\label{thm:088}
For every $n\geq0$, all patterns in $\class{088}$ are equidistributed
on $\mathcal I_n$. Moreover,
\[
F^{\mathcal I}_{\class{088}}(x,q)
=
\frac{\bigl(1+(1-q)x\bigr)\mathcal I(x)}
     {1+(1-q)x\mathcal I(x)}.
\]
\end{theorem}

\begin{proof}
For the representatives
\[
p=
\pattern{scale=0.6}{2}{1/1,2/2}{1/0,2/0,0/1,1/1,2/1,0/2}
\qquad\text{and}\qquad
p'=
\pattern{scale=0.6}{2}{1/1,2/2}{1/0,2/0,0/1,2/1,0/2,2/2},
\]
let $a<b$ be the selected values of an occurrence. The shadings force the elements to the left of $a$ to be exactly those smaller than 
$a$. Hence the elements to the left of $a$ are exactly
$1,\ldots,a-1$, so
$\pi_a=a$ and $\{\pi_1,\ldots,\pi_a\}=[a]$. For $p$, the remaining
shading forces $b=a+1$, whereas for $p'$ it forces $b=\pi_n$.
Thus, for every $\pi\in\mathcal I_n$,
$\occ_p(\pi)=\occ_{p'}(\pi)
=
\#\bigl\{
a\in[n-1]:
\{\pi_1,\ldots,\pi_a\}=[a],\ \pi_a=a
\bigr\}.$

Given $\sigma\in\mathfrak S_r$ and $\tau\in\mathfrak S_s$, their
direct sum is
$\sigma\oplus\tau
=
\sigma_1\cdots\sigma_r
(\tau_1+r)\cdots(\tau_s+r).$
Thus $\{\pi_1,\ldots,\pi_a\}=[a]$ means that $\pi$ has a
direct-sum cut after position $a$. Since $\pi_a=a$, we also have
$\{\pi_1,\ldots,\pi_{a-1}\}=[a-1]$, so $\pi_a=a$ forms a component of size 1 in the decomposition into sum-indecomposable components, which is called a singleton component.
For example, $21354=21\oplus1\oplus21$, where the element $3$ is the
singleton component.

Let $Y(x)$ be the generating function for nonempty
sum-indecomposable involutions. Since inversion preserves direct sums and sum indecomposability,
uniqueness of the direct-sum decomposition implies that every
direct-sum component of an involution is itself an involution. Hence
$\mathcal I(x)=\frac{1}{1-Y(x)}.$
Write a nonempty involution as
$\pi=\alpha_1\oplus\cdots\oplus\alpha_r$, where each $\alpha_i$ is
sum indecomposable. For $i<r$, a singleton $\alpha_i$ contributes $qx$, while the
non-singleton choices are counted by $Y(x)-x$. Thus each of
$\alpha_1,\ldots,\alpha_{r-1}$ is counted by
$Y(x)+(q-1)x$. Finally, the choices for $\alpha_r$ are counted by
$Y(x)$, since the last component contributes no occurrence. Hence, for $r\geq1$, involutions with exactly $r$ direct-sum
components are counted by
$\bigl(Y(x)+(q-1)x\bigr)^{r-1}Y(x).$
Summing over $r\geq1$ gives 
$$
F^{\mathcal I}_{\class{088}}(x,q)
=
1+\sum_{r\geq1}
\bigl(Y(x)+(q-1)x\bigr)^{r-1}Y(x)
=
1+\frac{Y(x)}{1-Y(x)-(q-1)x}.
$$
Finally, substituting $Y(x)=1-\mathcal I(x)^{-1}$ yields the formula for $F^{\mathcal I}_{\class{088}}(x,q)$ stated above.
\end{proof}

\begin{theorem}\label{thm:161}
For every $n\geq0$, all patterns in $\class{161}$ are equidistributed
on $\mathcal I_n$. Moreover,
\begin{equation}\label{eq:161-continued-fraction}
F^{\mathcal I}_{\class{161}}(x,q)
=
\cfrac{1}{
1-x-\cfrac{qx^2}{
1-qx-\cfrac{q(1+q)x^2}{
1-qx-\cfrac{q(1+2q)x^2}{
1-qx-\ddots}}}}.
\end{equation}
\end{theorem}

\begin{proof}
For the representatives
\[
p=\pattern{scale=0.6}{2}{1/2,2/1}{1/0,2/0}
\qquad\text{and}\qquad
p'=\pattern{scale=0.6}{2}{1/2,2/1}{1/0,1/1},
\]
let $b>a$ be the selected values from left to right in an occurrence.
Every element $b$ that is not a right-to-left minimum determines a
unique occurrence: for $p$, the shading forces $a$ to be the smallest
element to the right of $b$, while for $p'$ it forces $a$ to be the
leftmost element to the right of $b$ that is smaller than $b$.
Hence, for every $\pi\in\mathcal I_n$,
$\occ_p(\pi)=\occ_{p'}(\pi)=n-\rmin(\pi)$.

To obtain the continued fraction, we use the reciprocity with
$\class{033}$. By Theorem~\ref{thm:033}, for $n\geq1$,
\[
[x^n]F^{\mathcal I}_{\class{161}}(x,q)
=
q^{n-1}[x^n]F^{\mathcal I}_{\class{033}}(x,q^{-1}),
\]
since $n-\rmin(\pi)=n-1-(\rmin(\pi)-1)$. Summing over $n\geq1$
gives
\[
F^{\mathcal I}_{\class{161}}(x,q)
=
1+q^{-1}
\bigl(F^{\mathcal I}_{\class{033}}(qx,q^{-1})-1\bigr).
\]
Substituting $qx$ for $x$ and $q^{-1}$ for $q$ in
Theorem~\ref{thm:033} gives \eqref{eq:161-continued-fraction}.
\end{proof}

\begin{theorem}\label{thm:272}
For every $n\geq0$, all patterns in $\class{272}$ are equidistributed
on $\mathcal I_n$. Moreover,
\begin{equation}\label{eq:272-distribution}
F^{\mathcal I}_{\class{272}}(x,q)
=
\frac{
\bigl(1-q(q-1)x^2\bigr)\mathcal I(x)
+
(q-1)x
\bigl(
\mathcal P(x^2)-1+(q-1)x^2\mathcal P(x^2)
\bigr)
}{
1-(q^2-1)x^2\mathcal P(x^2)
}.
\end{equation}
\end{theorem}

\begin{proof}
Consider the
representatives
$$
p=
\pattern{scale=0.6}{2}{1/2,2/1}
{0/0,1/0,0/1,1/1,2/1,2/2}
\qquad\text{and}\qquad
p'=
\pattern{scale=0.6}{2}{1/2,2/1}
{0/0,1/0,0/1,2/1,0/2,2/2}.
$$
Suppose that $a>b$ forms an occurrence of either $p$ or $p'$, and
set $r=b$. For $p$, the shadings force $a=r+1$, while for $p'$
they force $a=\pi_1$. In both cases, the remaining shadings require
every value larger than $r$ to occur to the left of every value in
$[r]$, with $r$ being the leftmost among the values in $[r]$.
Conversely, every such $r$ determines a unique occurrence of $p$
and of $p'$, namely $(r+1,r)$ and $(\pi_1,r)$, respectively.

For $\pi\in\mathcal I_n$, let $\mathit{msc}(\pi)$ denote the number of such $r\in[n-1]$. Thus both representatives have exactly $\operatorname{msc}(\pi)$ occurrences.
For example, in the involution $\pi=45312$, the value $r=3$ is the
unique value satisfying these conditions: the values $4$ and $5$
lie to the left of all values in $[3]$, and $3$ is the leftmost of
$1,2,3$. Hence $\operatorname{msc}(45312)=1$. 

% Now suppose that $\pi$ is an involution. Note that inversion
% reverses a skew sum:
% $(\alpha\ominus\beta)^{-1}
% =
% \beta^{-1}\ominus\alpha^{-1}.$
% Hence, if
% $\pi=\alpha_1\ominus\cdots\ominus\alpha_k$ is its canonical skew
% decomposition, then
% $\pi^{-1}
% =\alpha_k^{-1}\ominus\cdots\ominus\alpha_1^{-1}$.
% Since $\pi=\pi^{-1}$ and the canonical skew decomposition is unique,
% we must have
% $\alpha_i=\alpha_{k+1-i}^{-1}$ for every $i$.
% This implies that all non-central components occur symmetrically as \emph{inverse pairs} of the form $(\alpha, \alpha^{-1})$.
% Consequently, every involution has the form
% $\alpha_1\ominus\cdots\ominus\alpha_m
% \ominus\gamma
% \ominus\alpha_m^{-1}\ominus\cdots\ominus\alpha_1^{-1},$
% where the central component $\gamma$ is either empty or is a
% skew indecomposable involution.

% Let $\mathcal S(x)$ be the generating function for nonempty skew indecomposable permutations, and put $\mathcal P(x)=\sum_{n\geq0}n!x^n$. Since every permutation is a sequence of such components, $\mathcal{P}(x) = (1 - \mathcal{S}(x))^{-1}$. An inverse pair of components $(\alpha,\alpha^{-1})$ has doubled length, giving the generating function $A(x)=\mathcal S(x^2)=1-\mathcal P(x^2)^{-1}$. The sequence of these paired outer and inner layers is thus generated by $(1-A(x))^{-1}$. Let $C(x)$ count the possibly empty central skew indecomposable involution. Assembling these paired layers around the center yields $\mathcal I(x)=C(x)/(1-A(x))$, meaning $C(x)=\mathcal I(x)/\mathcal P(x^2)$.

It remains to enumerate involutions according to $\operatorname{msc}$. Every
nonempty permutation has a unique decomposition
$\pi=\alpha_1\ominus\cdots\ominus\alpha_k,$
where each $\alpha_i$ is skew indecomposable. A value $r$ counted by $\operatorname{msc}(\pi)$ determines a skew cut immediately before the block occupied by the values in $[r]$. Since $r$ is the leftmost element of this block, its corresponding component begins with its maximum. A skew-indecomposable component of size greater than one cannot begin with its maximum, as that element would immediately split off. Thus, this component must be a singleton block. Conversely,
every singleton component $\alpha_i$ with $i\geq2$ gives such an
$r$. Therefore, $\operatorname{msc}(\pi)$ is exactly the number of singleton components among $\alpha_2,\ldots,\alpha_k$.

By Lemma~\ref{lem:paired-skew-decomposition}, an involution $\pi$ takes the symmetric form $\alpha_1\ominus\cdots\ominus\alpha_m \ominus\gamma \ominus\alpha_m^{-1}\ominus\cdots \ominus\alpha_1^{-1}$. Let $\mathcal S(x)$ be the generating function for nonempty skew indecomposable permutations. Since every permutation is a sequence of such components, we have $\mathcal{P}(x) = (1 - \mathcal{S}(x))^{-1}$. An inverse pair $(\alpha,\alpha^{-1})$ has doubled length, so its generating function is $A(x)=\mathcal S(x^2)=1-\mathcal P(x^2)^{-1}$. It then follows from the structural enumeration in Lemma~\ref{lem:paired-skew-decomposition} that the central component $\gamma$ is generated by $C(x)=\mathcal I(x)(1-A(x))=\mathcal I(x)/\mathcal P(x^2)$.

For the distribution, each singleton component except the first
contributes a factor $q$. There are two cases to consider,
according as $m=0$ or $m\geq1$.
\begin{itemize}
    \item \textbf{No inverse pairs ($m=0$):} The involution consists solely of the center. Being the first component, it is never counted by $\operatorname{msc}$ even if it is a singleton. Its contribution remains unchanged as $C(x)$.
    
    \item \textbf{One or more inverse pairs ($m \geq 1$):} The involution decomposes into an outermost pair, $m-1$ inner pairs, and a center.
    \begin{itemize}
        \item \emph{Outermost pair $(\alpha_1,\alpha_1^{-1})$:} The left block is the first overall (receiving no $q$), but the right block is counted (receiving $q$). The singleton pair $(1,1)$ shifts weight from $x^2$ to $qx^2$. Thus, the outermost pair contributes $A(x)+(q-1)x^2$.
        \item \emph{Inner pairs $(\alpha_i,\alpha_i^{-1})$, $2\leq i\leq m$:} Neither block is first, so both are counted if they are singletons. The weight shifts from $x^2$ to $q^2x^2$, yielding $A(x)+(q^2-1)x^2$ per inner pair.
        \item \emph{center:} Protected by the outermost pair, the center is no longer the first component. A singleton center shifts weight from $x$ to $qx$, contributing $C(x)+(q-1)x$.
    \end{itemize}
\end{itemize}
For fixed 
$m\geq 1$, the contribution is
$\bigl(A(x)+(q-1)x^2\bigr)
\bigl(A(x)+(q^2-1)x^2\bigr)^{m-1}
\bigl(C(x)+(q-1)x\bigr).$
Summing this over all $m\geq1$ and adding the $m=0$ case yields
\begin{equation}\label{eq:272-structural-gf}
\begin{aligned}
F^{\mathcal I}_{\class{272}}(x,q)
&=
C(x)+
\sum_{m\geq1}
\bigl(A(x)+(q-1)x^2\bigr)
\bigl(A(x)+(q^2-1)x^2\bigr)^{m-1}
\bigl(C(x)+(q-1)x\bigr)
\\
&=
C(x)+
\frac{A(x)+(q-1)x^2}
{1-A(x)-(q^2-1)x^2}
\bigl(C(x)+(q-1)x\bigr).
\end{aligned}
\end{equation}

Finally, substituting $A(x)=1-\mathcal P(x^2)^{-1}$ and $C(x)=\mathcal I(x)/\mathcal P(x^2)$ into \eqref{eq:272-structural-gf} and simplifying yields the formula for $F^{\mathcal I}_{\class{272}}(x,q)$ stated above.
\end{proof}

% ========================================

\begin{theorem}\label{lem:endpoint-indicators}
For every $n\geq0$, all patterns within each of the classes
$\class{045}$, $\class{089}$, and $\class{179}$ are equidistributed
on $\mathcal I_n$. Moreover,
\begin{align*}
F^{\mathcal I}_{\class{045}}(x,q)
&=
\bigl(q+(1-q)x^2\bigr)\mathcal I(x)
+(1-q)(1+x),\\
F^{\mathcal I}_{\class{089}}(x,q)
&=
\bigl(1+(q-1)x\bigr)\mathcal I(x)
+(1-q)x,\\
F^{\mathcal I}_{\class{179}}(x,q)
&=
\bigl(q+(1-q)x\bigr)\mathcal I(x)
+(1-q).
\end{align*}
\end{theorem}

\begin{proof}
For the representatives displayed below, the shaded cells force the
following numbers of occurrences for every $n\geq2$ and $\pi\in\In$.
\begin{center}
\renewcommand{\arraystretch}{1.3}
\begin{tabular}{ccc}
\toprule
class & representative patterns & occurrences\\
\midrule
$\class{045}$
&
\pattern{scale=0.6}{2}{1/1,2/2}{0/0,1/0,2/0,1/1,2/1}
\pspace
\pattern{scale=0.6}{2}{1/1,2/2}{0/0,0/1,1/1,2/1,0/2}
\pspace
\pattern{scale=0.6}{2}{1/1,2/2}{0/0,1/0,2/0,2/1,2/2}
\pspace
\pattern{scale=0.6}{2}{1/1,2/2}{0/0,2/0,0/1,2/1,2/2}
&
$\mathbf1_{\{\pi_1<n\}}$
\\
$\class{089}$
&
\pattern{scale=0.6}{2}{1/1,2/2}{0/0,1/0,2/0,0/1,1/1,2/1,0/2}
\pspace
\pattern{scale=0.6}{2}{1/1,2/2}{0/0,1/0,2/0,0/1,2/1,0/2,2/2}
&
$\mathbf1_{\{\pi_1=1\}}$
\\
$\class{179}$
&
\pattern{scale=0.6}{2}{1/2,2/1}{0/0,1/0,2/0,0/1,1/1}
\pspace
\pattern{scale=0.6}{2}{1/2,2/1}{0/0,1/0,2/0,0/1,0/2}
\pspace
\pattern{scale=0.6}{2}{1/2,2/1}{0/0,2/0,0/1,2/1,0/2}
\pspace
\pattern{scale=0.6}{2}{1/2,2/1}{0/0,0/1,1/1,2/1,0/2}
&
$\mathbf1_{\{\pi_1\neq1\}}$
\\
\bottomrule
\end{tabular}
\end{center}

We consider the three rows separately.

For $\class{045}$, the first, third, and fourth representatives have a unique
occurrence precisely when $\pi_n\neq1$, while the second does so precisely when
$\pi_1\neq n$. Since $\pi$ is an involution,
$\operatorname{pos}_{\pi}(1)=\pi_1,$
these conditions are equivalent. Hence all four representatives have  $\mathbf1_{\{\pi_1<n\}}$ occurrences. If $\pi_1=n$, then $(1,n)$ is a transposition of $\pi$, and deleting it gives an involution in $\mathcal I_{n-2}$. Thus $I_{n-2}$ involutions have no
occurrence and $I_n-I_{n-2}$ have one, proving the first formula.

For $\class{089}$, either representative has a unique occurrence precisely when
$\pi_1=1$. Deleting this fixed point gives an involution in
$\mathcal I_{n-1}$. Hence $I_{n-1}$ involutions have one occurrence and
$I_n-I_{n-1}$ have none, proving the second formula.

For $\class{179}$, the shading of each representative forces a unique
occurrence precisely when $\pi_1\neq1$. There are $I_{n-1}$
involutions satisfying $\pi_1=1$, obtained by adjoining the fixed
point $1$ to an involution in $\mathcal I_{n-1}$. Thus there are
$I_{n-1}$ involutions with no occurrence and $I_n-I_{n-1}$ with one
occurrence. Hence, summing over all lengths gives the third formula.
\end{proof}

\begin{theorem}\label{thm:114}
For every $n\geq0$, all patterns in $\class{114}$ are equidistributed
on $\mathcal I_n$. Moreover,
\[
F^{\mathcal I}_{\class{114}}(x,q)
=
\bigl(1+(q-1)x^2\bigr)\mathcal I(x).
\]
\end{theorem}

\begin{proof}
For $n=2$, each representative has distribution $1+q$. Now let $n\ge3$. Each representative below occurs at most once. An occurrence is exactly one of the specific cycles below:
\begin{center}
\renewcommand{\arraystretch}{1.5}
\begin{tabular}{cl}
\toprule
representative patterns & forced cycles in $\pi$\\
\midrule
\pattern{scale=0.6}{2}{1/1,2/2}
{0/0,1/0,2/0,0/1,1/1,2/1,0/2,1/2}
& $(1)(2)$\\

\pattern{scale=0.6}{2}{1/1,2/2}
{0/0,1/0,2/0,0/1,1/1,2/1,2/2}
& $(2\,n)$\\

\pattern{scale=0.6}{2}{1/1,2/2}
{0/0,1/0,2/0,0/1,2/1,0/2,1/2,2/2}
& $(1)(n)$\\

\pattern{scale=0.6}{2}{1/2,2/1}
{0/0,1/0,2/0,1/1,2/1,0/2}
\pspace
\pattern{scale=0.6}{2}{1/2,2/1}
{0/0,1/0,2/0,0/1,1/1,2/1,0/2}
\pspace
\pattern{scale=0.6}{2}{1/2,2/1}
{0/0,1/0,2/0,1/1,2/1,0/2,1/2}
\pspace
\pattern{scale=0.6}{2}{1/2,2/1}
{0/0,1/0,2/0,0/1,1/1,2/1,0/2,1/2}
& $(1\,2)$\\

\pattern{scale=0.6}{2}{1/2,2/1}
{0/0,1/0,2/0,0/1,1/1,2/1,2/2}
\pspace
\pattern{scale=0.6}{2}{1/2,2/1}
{0/0,1/0,2/0,2/1,0/2,2/2}
\pspace
\pattern{scale=0.6}{2}{1/2,2/1}
{0/0,1/0,2/0,0/1,2/1,0/2,2/2}
\pspace
\pattern{scale=0.6}{2}{1/2,2/1}
{0/0,1/0,2/0,0/1,2/1,0/2,1/2,2/2}
& $(1\,n)$\\
\bottomrule
\end{tabular}
\end{center}
Each pattern restricts exactly two elements. Deleting them leaves an arbitrary involution on the remaining $n-2$ elements. Thus, exactly $I_{n-2}$ involutions contain one occurrence, and the remaining $I_n-I_{n-2}$ contain none. Therefore, for all $p \in \class{114}$,
\[
\sum_{\pi\in\mathcal I_n}q^{\operatorname{occ}_p(\pi)}
=(I_n-I_{n-2})+I_{n-2}q.
\]
Summing this formula over $n \ge 2$ yields the distribution for $F^{\mathcal I}_{\class{114}}(x,q)$ stated above.
\end{proof}

% \begin{theorem}\label{thm:154}
% All patterns in $\class{154}$ are equidistributed.
% Moreover,
% \[
% [x^2]F^{\mathcal I}_{\class{154}}(x,q)=1+q,
% \qquad
% [x^n]F^{\mathcal I}_{\class{154}}(x,q)=I_n
% \quad(n\geq3).
% \]
% \end{theorem}

\begin{theorem}\label{thm:154}
For every $n\geq0$, all patterns in $\class{154}$ are equidistributed
on $\mathcal I_n$. Moreover,
\[
F^{\mathcal I}_{\class{154}}(x,q)
=
\mathcal I(x)+(q-1)x^2.
\]
\end{theorem}

\begin{proof}
The following four representatives can occur only when $n=2$:
\[
\begin{gathered}
p_1=
\pattern{scale=0.6}{2}{1/1,2/2}{0/0,1/0,2/0,0/1,1/1,2/1,0/2,1/2,2/2},
\qquad
p_2=
\pattern{scale=0.6}{2}{1/2,2/1}{0/0,1/0,2/0,1/1,2/1,0/2,2/2},
\qquad
p_3=
\pattern{scale=0.6}{2}{1/2,2/1}{0/0,1/0,2/0,0/1,1/1,2/1,0/2,2/2},
\qquad
p_4=
\pattern{scale=0.6}{2}{1/2,2/1}{0/0,1/0,2/0,0/1,1/1,2/1,0/2,1/2,2/2}.
\end{gathered}
\]

If an occurrence of $p_1$ or $p_4$ exists, the shading of all nine cells prohibits any unselected elements, immediately forcing $n=2$.
For $p_2$, the shading and involutivity jointly force the selected elements to be $(1,n)$ and $(n,1)$. Since the central cell is also shaded, no other elements can exist, yielding $n=2$.
If an occurrence of $p_3$ exists, the shading forces the selected elements to be $(1,2)$ and $(n,1)$. Involutivity on $(1,2)$ yields $(2,1)$, forcing $(2,1) = (n,1)$ and thus $n=2$.

Thus, for $n \ge 3$, these patterns never occur, giving the distribution $I_n$. For $n=2$, each pattern occurs in exactly one of the two involutions, yielding $1+q$. Thus, all four representatives are equidistributed.
\end{proof}

\begin{theorem}\label{thm:085}
For every $n\geq0$, all patterns in $\class{085}$ are equidistributed
on $\mathcal I_n$. Moreover,
\begin{equation}\label{eq:085-generating-function}
F^{\mathcal I}_{\class{085}}(x,q)
=
\mathcal I(x)
+(q-1)\left(
x\bigl(\mathcal I(x)-1\bigr)
+
\sum_{j\geq2}(j-2)!x^{2j}
\sum_{r\geq0}
\binom{r+j-1}{j-1}I_r x^r
\right).
\end{equation}
\end{theorem}

\begin{proof}
For the representatives
\[
p=
\pattern{scale=0.6}{2}{1/1,2/2}
{0/0,1/0,0/1,1/1,2/1,0/2}
\quad\text{and}\quad
p'=
\pattern{scale=0.6}{2}{1/1,2/2}
{0/0,1/0,2/0,0/1,2/1,2/2},
\]
each pattern has at most one occurrence. For $p$, the shading forces
the selected elements to $(1,a)$ and $(j,a+1)$. When \(a=1\), the element \(1\) is fixed, and deleting it leaves an arbitrary involution on the remaining \(n-1\) elements. Hence this case contributes \(I_{n-1}\).

Let $a>1$. Involutivity forces the transpositions $(1\,a)$ and
$(j\,\,a+1)$. The shading requires all elements in positions
$2,\ldots,j-1$ to exceed $a+1$; since $\pi_a=1$, this forces
$j<a$. The $j-2$ positions $2,\ldots,j-1$ must pair bijectively
with elements chosen from $a+2,\ldots,n$. Together with
$(1\,a)$ and $(j\,\,a+1)$, this fixes exactly $2j$ elements,
leaving an arbitrary involution on the remaining $n-2j$ elements.
Thus, for fixed $a$ and $j$, there are
\[
(j-2)!\binom{n-a-1}{j-2}I_{n-2j}
\]
choices. Here
$2\leq j\leq\lfloor n/2\rfloor,$
$j+1\leq a\leq n-j+1.$
By the hockey-stick identity,
\[
\sum_{a=j+1}^{n-j+1}
\binom{n-a-1}{j-2}
=
\binom{n-j-1}{j-1}.
\]
Hence the number of involutions containing $p$ is $E_0^{085}=E_1^{085}=0$, and for $n\ge2$,
\[
E_n^{085}
=
I_{n-1}
+
\sum_{j=2}^{\lfloor n/2\rfloor}
(j-2)!\binom{n-j-1}{j-1}I_{n-2j}.
\]

For $p'$, an occurrence is uniquely determined by
$a=\pi_1$ and $b=\pi_n$. The case $a=1$ again contributes
$I_{n-1}$. For $a>1$, involutivity forces $(1\,a)$ and $(b\,n)$,
and the $b-2$ positions $2,\ldots,b-1$ must pair with elements
chosen from $a+1,\ldots,n-1$. Thus, for fixed $a$ and $b$, there
are
\[
(b-2)!\binom{n-a-1}{b-2}I_{n-2b}
\]
choices. Setting $b=j$ gives the same summand and the same bounds
as for $p$. Hence $p'$ also occurs in exactly $E_n^{085}$
involutions.

Since either pattern occurs at most once,
\[
[x^n]F^{\mathcal I}_{\class{085}}(x,q)
=
(I_n-E_n^{085})+qE_n^{085}
=
I_n+(q-1)E_n^{085}.
\]
Moreover, summing over $n\geq2$ and setting $r=n-2j$ gives
\[
\sum_{n\geq2}E_n^{085}x^n
=
x\bigl(\mathcal I(x)-1\bigr)
+
\sum_{j\geq2}(j-2)!x^{2j}
\sum_{r\geq0}
\binom{r+j-1}{j-1}I_rx^r.
\]
Therefore
$F^{\mathcal I}_{\class{085}}(x,q)
=
\mathcal I(x)
+
(q-1)\sum_{n\geq2}E_n^{085}x^n,$
which gives \eqref{eq:085-generating-function}.
\end{proof}

\begin{theorem}\label{thm:250}
For every $n\geq0$, all patterns in $\class{250}$ are equidistributed
on $\mathcal I_n$. 
Moreover,
\begin{equation}\label{eq:E250}
F^{\mathcal I}_{\class{250}}(x,q)
=
\bigl(1+(q-1)x^2\mathcal P(x^2)\bigr)\mathcal I(x).
\end{equation}
\end{theorem}

\begin{proof}
For the representatives $p$ and $p'$ given by
\[
p=
\pattern{scale=0.6}{2}{1/2,2/1}{0/0,2/0,1/1,2/1,0/2,1/2}
\qquad\text{and}\qquad
p'=
\pattern{scale=0.6}{2}{1/2,2/1}{0/0,1/0,2/0,0/1,1/1,2/2},
\]

Let the inversion $b>a$ occur at positions $i<j$. For $\pi$,
the shading forces the values between $a$ and $b$ to appear before position $i$, yielding $i=b-a$.
Furthermore, counting the remaining elements constrained by the shaded regions to the left and right yields the second position index $j = b$. By involutivity, $\pi_i=b$ and $\pi_b=a$ then reduce to $i=a$ and $b=2a$. The intervals $[1,a-1]$ and $[a+1,2a-1]$ can be paired by an arbitrary bijection, $a$ pairs with $2a$, and the remaining $n-2a$ elements support an arbitrary involution, yielding $(a-1)!I_{n-2a}$ involutions. Distinct values of $a$ are mutually incompatible, so $p$ occurs at most once.

For $p'$, the smaller value is $1$. Let $j=\pi_1$ and let $b$ be the larger selected value. The shading forces $b+j=n+2$ with $b\le j$. The sets $[2,b-1]$ and $[j+1,n]$ pair arbitrarily, while $[b,j-1]$ supports an arbitrary involution. Setting $a=b-1$, this counts $(b-2)!I_{j-b} = (a-1)!I_{n-2a}$ involutions. Since $j$ uniquely determines $b$, $p'$ also occurs at most once. 

Let
$E_n^{250}
=
\sum_{a=1}^{\lfloor n/2\rfloor}
(a-1)!I_{n-2a}.$
Since each pattern occurs at most once, there are $I_n-E_n^{250}$
involutions avoiding the pattern and $E_n^{250}$ involutions containing it
exactly once. Therefore,
$F^{\mathcal I}_{\class{250}}(x,q)
=
\sum_{n\ge0}{(I_n+(q-1)E_n^{250})x^n}$.
% Interchanging the order of summation and setting $m=n-2a$, we obtain
% \begin{align*}
% \sum_{n\ge0}E_nx^n
% &=\sum_{n\ge0}\sum_{a=1}^{\lfloor n/2\rfloor}(a-1)!I_{n-2a}x^n\\
% &=\sum_{a\ge1}(a-1)!x^{2a}\sum_{m\ge0}I_mx^m\\
% &=\mathcal I(x)x^2\sum_{r\ge0}r!(x^2)^r
% =x^2\mathcal P(x^2)\mathcal I(x),
% \end{align*}
Interchanging the order of summation, we obtain
$$\sum_{n\ge0}E_n^{250}x^n
=\sum_{n\ge0}\sum_{a=1}^{\lfloor n/2\rfloor}(a-1)!I_{n-2a}x^n
=x^2\mathcal P(x^2)\mathcal I(x).$$ Therefore,
\[
F^{\mathcal I}_{\class{250}}(x,q)
=\sum_{n\ge0}\bigl(I_n+(q-1)E_n^{250}\bigr)x^n
=\bigl(1+(q-1)x^2\mathcal P(x^2)\bigr)\mathcal I(x).
\qedhere
\]

\end{proof}

\section{Equidistributions via diagonal symmetry}
\label{sec:involution-plot-symmetry}

In this section, we use the symmetry of involution plots about the
main diagonal to prove four further equidistributions.

\begin{theorem}\label{thm:176}
For every $n\geq0$, all patterns in $\class{176}$ are equidistributed
on $\mathcal I_n$. Moreover,
\begin{equation}\label{eq:176-continued-fraction}
F^{\mathcal I}_{\class{176}}(x,q)
=
\cfrac{1}{
1-x-\cfrac{qx^2}{
1-x-\cfrac{(q+1)x^2}{
1-x-\cfrac{(q+2)x^2}{
1-x-\ddots}}}}.
\end{equation}
\end{theorem}

\begin{proof}
For the representatives
\[
p=
\pattern{scale=0.6}{2}{1/2,2/1}{2/0,0/1,1/1}
\qquad\text{and}\qquad
p'=
\pattern{scale=0.6}{2}{1/2,2/1}{1/0,2/0,0/2},
\]
we show that any involution contains the same number of occurrences of $p$ and $p'$. Let $\pi\in\In$. For $p$, suppose that $ab$, with $a>b$, is an occurrence. The shading forces $b$ to be a right-to-left minimum having a larger element to its left. Moreover, the element $a$ is uniquely determined as the smallest element to the left of $b$ that is larger than $b$. Thus each right-to-left minimum having a larger element to its left contributes exactly one occurrence of $p$. Similarly, for $p'$, the element $a$ in an occurrence $ab$ must be a left-to-right maximum having a smaller element to its right. The element $b$ is uniquely determined as the smallest element to the right of $a$. Thus each left-to-right maximum having a smaller element to its right contributes exactly one occurrence of $p'$. 

Since $\pi$ is an involution, its plot is symmetric about the main diagonal. Reflection in the main diagonal sends a left-to-right maximum having a smaller element to its right to a right-to-left minimum having a larger element to its left, and vice versa. Hence $\pi$ contains the same number of occurrences of $p$ and $p'$.
Thus, for every $\pi\in\mathcal I_n$,
\[
\occ_p(\pi)=\occ_{p'}(\pi)
=
\#\{b:\ b\text{ is a right-to-left minimum having a larger
element to its left}\}.
\]

It remains to prove \eqref{eq:176-continued-fraction}. Let $(a\ b)$
be a $2$-cycle of $\pi$, with $a<b$, so that $\pi_b=a$. The element
$a$ is a right-to-left minimum if and only if there is no other
$2$-cycle $(c\ d)$ such that $c<a<b<d$. Indeed, such a cycle gives
$\pi_d=c<a=\pi_b$. Conversely, if $a$ is not a right-to-left
minimum, then some $d>b$ satisfies $\pi_d=c<a$, and involutivity
gives the $2$-cycle $(c\ d)$ with $c<a<b<d$.
Moreover, the larger element $b$ lies to the left of
$a$. Hence the common number of occurrences equals the number of $2$-cycles whose arcs are not contained in another arc.

Under Biane's bijection between involutions and labeled Motzkin paths, a $2$-cycle $(a,b)$, with $a<b$, is represented by a down step. The label of this down step is counted from $1$: it is equal to
$1+\#\{(c,d): c<a<b<d\}.$
Thus label $1$ occurs exactly when no arc properly contains $(a,b)$. Therefore the $2$-cycles counted above are exactly the down steps of label $1$. Barnabei et al.~\cite[Theorem~3.5]{BarnabeiMaxima} give the continued fraction in which the variable $y$ records down steps of label $1$, while $t$ records horizontal steps at height $0$. Setting $y=q$ and $t=1$ gives \eqref{eq:176-continued-fraction}.
\end{proof}

% \begin{remark}\label{rem:176-recurrence}
% Let $k_n(q)=[x^n]F^{\mathcal I}_{\class{176}}(x,q)$ be the
% distribution polynomial for $\class{176}$ on $\mathcal I_n$.
% Besides the continued fraction in Theorem~\ref{thm:176}, these
% polynomials satisfy $k_0(q)=k_1(q)=1$ and, for $n\geq2$,
% $$
% k_n(q)
% =
% k_{n-1}(q)
% +(n-1)k_{n-2}(q)
% +(q-1)\sum_{j=0}^{n-2}
% k_j(q)k_{n-2-j}(q).
% $$
% Equivalently, if
% $F=F^{\mathcal I}_{\class{176}}(x,q)$, then this recurrence is
% obtained by comparing coefficients in the identity
% $F=1+xF+x^2F+x^3\frac{\partial F}{\partial x}+(q-1)x^2F^2$.
% In particular, setting $q=1$ reduces the recurrence to
% $k_n(1)=k_{n-1}(1)+(n-1)k_{n-2}(1)$, the usual recurrence for
% involution numbers.
% \end{remark}

\begin{lemma}
\label{lem:transpose-invariance}
Let $T=\{(1,0),(2,0),(2,1),(0,2)\}$, and let $p=(21,R)$ be a
mesh pattern with $T\subseteq R$.
If $(i,j)\in R$, then $p$ and
$
(21,R\cup\{(j,i)\})
$
have the same occurrences in every involution.
\end{lemma}

\begin{proof}
Let
$
p'=(21,R\cup\{(j,i)\}),
$
where $(i,j)\in R$. Clearly, every occurrence of $p'$ is an occurrence of $p$. Conversely, let \(\pi\in\mathcal I_n\), and suppose that positions $r<s$, with values
$a=\pi_r>b=\pi_s$, form an occurrence of $p$.
Since $\pi$ is
an involution, $(a,r)$ and $(b,s)$ also belong to its plot. We claim
that $b=r$.
Suppose first that $b<r$. If $s>a$, then $(b,s)$ lies in the shaded
cell $(0,2)$; if $s<a$, then $(a,r)$ lies in the shaded cell $(2,1)$;
and if $s=a$, involutivity gives $b=r$. Thus $b<r$ is impossible.
Similarly, if $b>r$, then $(a,r)$ lies in the shaded cell $(1,0)$
when $a<s$, and in $(2,0)$ when $a>s$; if $a=s$, involutivity again
gives $b=r$. Hence $b=r$.

Since $\pi$ is an involution, $b=r$ also gives $a=s$. Thus the two
selected points are $(r,s)$ and $(s,r)$, so reflection in the main
diagonal interchanges the cells $(i,j)$ and $(j,i)$.
Since $(i,j)$ is shaded, the corresponding cell contains no point of the plot of $\pi$. By symmetry about the main diagonal, the transposed cell $(j,i)$ also contains no point of the plot. Hence the given occurrence of $p$ is also an occurrence of $p'$.
\end{proof}

\begin{theorem}\label{thm:forcing-classes}
For every $n\geq0$, all patterns within each of the classes
$\class{218}$, $\class{219}$, and $\class{228}$ are equidistributed
on $\mathcal I_n$. Moreover,
\begin{align}
F^{\mathcal I}_{\class{218}}(x,q)
&=
\frac{\mathcal I(x)}
     {1-(q-1)x^2\mathcal I(x)^2},
\label{eq:218-distribution}\\
F^{\mathcal I}_{\class{219}}(x,q)
&=
\mathcal I(x)+(q-1)x^2\mathcal I(x)^2,
\label{eq:219-distribution}\\
F^{\mathcal I}_{\class{228}}(x,q)
&=
\frac{\mathcal I(x)}
     {1-(q-1)x^2\mathcal I(x)}.
\label{eq:228-distribution}
\end{align}
\end{theorem}

\begin{proof}
For the representatives displayed in the following rows, we prove the
stronger pointwise statement that the occurrence sets agree on every
involution:
\[
\begin{array}{c|l}
\class{218}
&
\pattern{scale=0.6}{2}{1/2,2/1}{1/0,2/0,2/1,0/2}
\pspace
\pattern{scale=0.6}{2}{1/2,2/1}{1/0,2/0,0/1,2/1,0/2}
\pspace
\pattern{scale=0.6}{2}{1/2,2/1}{1/0,2/0,0/1,2/1,0/2,1/2}
\\[2mm]

\class{219}
&
\pattern{scale=0.6}{2}{1/2,2/1}{0/0,1/0,2/0,2/1,0/2}
\pspace
\pattern{scale=0.6}{2}{1/2,2/1}{0/0,1/0,2/0,0/1,2/1,0/2}
\pspace
\pattern{scale=0.6}{2}{1/2,2/1}{0/0,1/0,2/0,2/1,0/2,1/2}
\pspace
\pattern{scale=0.6}{2}{1/2,2/1}{0/0,1/0,2/0,0/1,2/1,0/2,1/2}
\\[2mm]

\class{228}
&
\pattern{scale=0.6}{2}{1/2,2/1}{1/0,2/0,1/1,2/1,0/2}
\pspace
\pattern{scale=0.6}{2}{1/2,2/1}{1/0,2/0,0/1,1/1,2/1,0/2}
\pspace
\pattern{scale=0.6}{2}{1/2,2/1}{1/0,2/0,0/1,1/1,2/1,0/2,1/2}.
\end{array}
\]
In each row, the shading of the first representative contains the set
$T$ of Lemma~\ref{lem:transpose-invariance}, and every other
representative is obtained by adding one or both of the cells $(0,1)$
and $(1,2)$, the transposes of the already shaded cells $(1,0)$ and
$(2,1)$. Lemma~\ref{lem:transpose-invariance} therefore shows that
these additions do not change the occurrence set on involutions.
Thus the representatives in each row have identical occurrence sets.

It remains to determine the three distributions. Let $Y(x)$ be the
ordinary generating function for a nonempty sum-indecomposable
involution. Since every involution is uniquely a direct sum of such
components, $\mathcal I(x)=1/(1-Y(x))$, so
\[
Y(x)=1-\mathcal I(x)^{-1},
\qquad
D_{218}(x)=x^2\mathcal I(x),
\qquad
D_{228}(x)=x^2.
\]
Here $D_{218}(x)$ and $D_{228}(x)$ enumerate the components that
produce an occurrence.

For $\class{218}$, Lemma~\ref{lem:transpose-invariance} forces the selected
positions $i<j$ to form the $2$-cycle $(i\,j)$. The shading and its
reflection imply that all elements before $i$ are smaller than $i$,
all elements between $i$ and $j$ have values strictly between $i$ and
$j$, and all elements after $j$ are larger than $j$. Hence the
interval $[i,j]$ is a single direct-sum component. After
standardization, this component is obtained by surrounding an
arbitrary involution $\tau$ with an outer $2$-cycle; explicitly, it
has the form
$(m+2)(\tau_1+1)\cdots(\tau_m+1)1$ for some
$\tau\in\mathcal I_m$. Conversely, every such component gives
exactly one occurrence. This explains
$D_{218}(x)=x^2\mathcal I(x)$.

For $\class{219}$, the extra shaded cell $(0,0)$ forces $i=1$.
Consequently, the distinguished component must be the first
direct-sum component, and it is again of the type counted by
$D_{218}(x)$. Hence an involution has at most one occurrence in this
class.

For $\class{228}$, the extra shaded cell $(1,1)$ forces $j=i+1$.
Thus the distinguished direct-sum component is exactly $21$, giving
$D_{228}(x)=x^2$.

We now mark each distinguished component by $q$. If components of a
certain type have generating function $D(x)$, replacing their weight
$1$ by $q$ changes the component series from $Y(x)$ to
$Y(x)+(q-1)D(x)$. Therefore, for $\class{218}$ and $\class{228}$ we
simply form an arbitrary sequence of weighted components. For
$\class{219}$ only the first component is marked, while the remaining
components form an arbitrary involution. Hence
\[
\begin{aligned}
F^{\mathcal I}_{\class{218}}(x,q)
&=
\frac{1}{1-Y(x)-(q-1)D_{218}(x)}
=
\frac{\mathcal I(x)}
     {1-(q-1)x^2\mathcal I(x)^2},
\\
F^{\mathcal I}_{\class{219}}(x,q)
&=
1+\bigl(Y(x)+(q-1)D_{218}(x)\bigr)\mathcal I(x)
=
\mathcal I(x)+(q-1)x^2\mathcal I(x)^2,
\\
F^{\mathcal I}_{\class{228}}(x,q)
&=
\frac{1}{1-Y(x)-(q-1)D_{228}(x)}
=
\frac{\mathcal I(x)}
     {1-(q-1)x^2\mathcal I(x)}.
\end{aligned}
\]
This completes the proof.
\end{proof}

% ============================

\section{Equidistributions via bijections}
\label{sec:involution-preserving-bijections}

In this section, the required equidistributions are proved by explicit
bijections on involutions.

% ==============================

\subsection{Conjugation Bijections for Classes INV-064 and INV-213}

For $n\geq1$, let $g_n\in\In$ be defined by
\[
g_n(k) =
\begin{cases}
  1, & \text{if } k=1, \\
  n+2-k, & \text{if } 2 \leq k \leq n.
\end{cases}
\]
We define $\Gamma_n\colon\In\to\In$ by $\Gamma_n(\pi)=g_n\pi g_n^{-1}=g_n\pi g_n$.
For every $\pi\in\In$, we have $\Gamma_n(\pi)^{-1} = g_n\pi^{-1}g_n = g_n\pi g_n = \Gamma_n(\pi)$, so $\Gamma_n(\pi)$ is an involution.

\begin{theorem}\label{thm:064}
For every $n\geq0$, all patterns in $\class{064}$ are equidistributed
on $\mathcal I_n$. Moreover, with $R(x,q)$ as in
Lemma~\ref{lem:R-recurrence},
\[
F^{\mathcal I}_{\class{064}}(x,q)
=
(1-x)\mathcal I(x)+xR(x,q).
\]
\end{theorem}

\begin{proof}
For the representatives
\[
p=
\pattern{scale=0.6}{2}{1/1,2/2}
{0/0,1/0,2/0,0/1,1/1,0/2}
\qquad\text{and}\qquad
p'=
\pattern{scale=0.6}{2}{1/1,2/2}
{0/0,1/0,2/0,0/1,0/2,2/2},
\]
fix $n\geq1$. The shaded bottom row and left column force the first
selected element to be $(1,1)$. Thus an occurrence can exist only
when $\pi_1=1$.

Note that $\Gamma_n(\pi)=g_n\pi g_n$, where $g_n$ fixes $1$ and
reverses $[2,n]$. Hence $\Gamma_n$ fixes the point $(1,1)$ and acts
as a $180^\circ$ rotation on the subgrid
$[2,n]\times[2,n]$. Under this rotation, the shaded cell $(1,1)$
of $p$ is carried to the shaded cell $(2,2)$ of $p'$, while the
remaining shading conditions are preserved. Since
$\Gamma_n^2=\mathrm{id}$, it follows that
$\occ_p(\pi)=\occ_{p'}(\Gamma_n(\pi)).$

It remains to determine their common distribution. If $\pi_1\ne1$,
then $p$ has no occurrence. If $\pi_1=1$, write
$\pi=1\oplus\sigma$ with $\sigma\in\mathcal I_{n-1}$. Pairing the
initial $1$ with an element of $\sigma$ gives an occurrence of $p$
precisely when that element is a left-to-right minimum of $\sigma$.
Hence
$\occ_p(\pi)=\lmin(\sigma)$.
There are $I_n-I_{n-1}$ involutions in $\mathcal I_n$ with
$\pi_1\ne1$, while those with $\pi_1=1$ are in bijection with
$\mathcal I_{n-1}$. Therefore, for $n\geq1$,
\[
[x^n]F^{\mathcal I}_{\class{064}}(x,q)
=
I_n-I_{n-1}+R_{n-1}(q).
\]
Summing over $n\geq1$ yields the stated generating function.
\end{proof}

% \begin{theorem}\label{thm:213}
% For every $n\geq0$, all patterns in $\class{213}$ are equidistributed
% on $\mathcal I_n$. Moreover,
% \[
% F^{\mathcal I}_{\class{213}}(x,q)
% =
% \mathcal I(x)+(q-1)\mathcal A^{\mathcal I}(x),
% \]
% where
% \[
% \mathcal A^{\mathcal I}(x)
% =
% \sum_{n\geq2}A_n^{\mathcal I}x^n,
% \qquad
% A_n^{\mathcal I}
% =
% \sum_{k=0}^{\lfloor (n-2)/2\rfloor}
% k!\binom{n-k-2}{k}I_{n-2k-2}.
% \]
% \end{theorem}

\begin{theorem}\label{thm:213}
For every $n\geq0$, all patterns in $\class{213}$ are equidistributed
on $\mathcal I_n$. Moreover,
\[
F^{\mathcal I}_{\class{213}}(x,q)
=
\mathcal I(x)
+
(q-1)[u^2]R(x,u),
\]
where $R(x,u)$ is defined in Lemma~\ref{lem:R-recurrence}.
\end{theorem}

\begin{proof}
Consider the representatives
\[
p=
\pattern{scale=0.6}{2}{1/2,2/1}
{0/0,1/0,2/0,0/1,1/1,0/2}
\qquad\text{and}\qquad
p'=
\pattern{scale=0.6}{2}{1/2,2/1}
{0/0,1/0,2/0,0/1,0/2,2/2}.
\]
We first prove their equidistribution. The case $n=0$ is trivial, so
fix $n\ge1$. If $(1,j)$ and $(j,1)$ form an occurrence of $p$ in
$\pi\in\In$, then the shaded bottom row and left column force
$\pi_1=j$ and $\pi_j=1$. Under $\Gamma_n$, these two points are sent to
$(1,g_n(j))$ and $(g_n(j),1)$, respectively. Since
$g_n(j)\in[2,n]$, the transformed pair is again order-isomorphic to
$21$. As in the proof of Theorem~\ref{thm:064}, $\Gamma_n$ sends the
shaded cell $(1,1)$ of $p$ to the shaded cell $(2,2)$ of $p'$ and
preserves all remaining shading conditions. Since
$\Gamma_n^2=\mathrm{id}$, it induces a bijection between the occurrences
of $p$ in $\pi$ and those of $p'$ in $\Gamma_n(\pi)$. Hence
$\occ_p(\pi)=\occ_{p'}(\Gamma_n(\pi)).$

We now determine the common distribution using $p$. If $\pi_1=j>1$,
then involutivity gives $\pi_j=1$, so the only possible occurrence is
formed by $(1,j)$ and $(j,1)$. The shaded cell $(1,1)$ is empty
precisely when $\pi_i>j$ for $2\le i<j$, which is equivalent to
$\pi_1=j$ and $\pi_j=1$ being the only left-to-right minima of $\pi$.
Therefore
$\occ_p(\pi)=\mathbf1_{\{\lmin(\pi)=2\}}.$

Since every involution contains at most one occurrence of $p$,
\[
\sum_{\pi\in\In}q^{\occ_p(\pi)}
=
I_n
+
(q-1)\#\{\pi\in\In:\lmin(\pi)=2\}.
\]
Summing over $n\geq0$ gives
$F^{\mathcal I}_{\class{213}}(x,q)
=
\mathcal I(x)+(q-1)[u^2]R(x,u),$
as required.
\end{proof}

% It remains to count such involutions. Fix $\pi_1=j$. Each of the
% elements $2,\ldots,j-1$ must be paired with a distinct element of
% $\{j+1,\ldots,n\}$. Choosing these $j-2$ partners and matching them
% with $2,\ldots,j-1$ gives
% $(j-2)!\binom{n-j}{j-2}$ possibilities, while the remaining
% $n-2j+2$ elements form an arbitrary involution. Hence
% \[
% A_n^{\mathcal I}
% =
% \sum_{j=2}^{\lfloor (n+2)/2\rfloor}
% (j-2)!\binom{n-j}{j-2}I_{n-2j+2}
% =
% \sum_{k=0}^{\lfloor (n-2)/2\rfloor}
% k!\binom{n-k-2}{k}I_{n-2k-2}.
% \]
% Since the common occurrence number is $0$--$1$ valued, summing over
% $n$ gives
% \[
% F^{\mathcal I}_{\class{213}}(x,q)
% =
% \mathcal I(x)+(q-1)\mathcal A^{\mathcal I}(x).
% \qedhere
% \]
% \end{proof}

% ================================

\subsection{RSK Bijections for Classes INV-046 and INV-052}

By the RSK correspondence
\cite{Robinson1938,Schensted1961,Knuth1970}, if $\pi\in\In$, then
$\operatorname{RSK}(\pi)=(T,T)$ for a
standard Young tableau $T$. Define
$\Theta_n\colon\In\to\In$ by
\[
\operatorname{RSK}(\Theta_n(\pi))
=
(T^{\mathsf t},T^{\mathsf t}).
\]
Recall that if $\operatorname{RSK}(\sigma)=(P,Q)$, then
$\operatorname{RSK}(\sigma^{-1})=(Q,P)$. Hence, $\Theta_n$ maps
$\In$ into itself, and $(T^{\mathsf t})^{\mathsf t}=T$ shows that it
is an involutive bijection of $\In$.

By the RSK descent correspondence and tableau
transposition (see \cite[Proposition~1]{Chow}),
\begin{equation}\label{eq:RSK-descent-complement}
\Des(\Theta_n(\pi))
=
\Des(T^{\mathsf t})
=
[n-1]\setminus\Des(T)
=
[n-1]\setminus\Des(\pi).
\end{equation}
Consequently,
$\operatorname{des}(\Theta_n(\pi))=\operatorname{asc}(\pi)$ and
$\operatorname{asc}(\Theta_n(\pi))=\operatorname{des}(\pi)$.
Thus transposing the RSK recording tableau gives a bijection on involutions that interchanges ascents and descents.

\begin{theorem}\label{thm:046}
For every $n\geq0$, all patterns in $\class{046}$ are equidistributed
on $\mathcal I_n$. Moreover,
\begin{equation}\label{eq:046-generating-function}
F^{\mathcal I}_{\class{046}}(x,q)
=
(1-q)\sum_{r\geq0}
\frac{q^r}
{\bigl(1-(1-q)x\bigr)^{r+1}
 \bigl(1-(1-q)^2x^2\bigr)^{\binom{r+1}{2}}}.
\end{equation}
\end{theorem}

\begin{proof}
For the representatives
\[
p=
\pattern{scale=0.6}{2}{1/1,2/2}{0/1,1/1,2/1}
\qquad\text{and}\qquad
p'=
\pattern{scale=0.6}{2}{1/2,2/1}{0/1,1/1,2/1},
\]

The shaded middle row is empty if and only if the two selected values
are consecutive. Let $a<b$ be the selected positions. For $p$, we
have $\pi_b=\pi_a+1$. Setting $k=\pi_a$ and using involutivity gives
$\pi_k=a<b=\pi_{k+1}$, so each occurrence of $p$ corresponds to an
ascent of $\pi$, and conversely. Hence,
$\occ_p(\pi)=\operatorname{asc}(\pi)$. For $p'$, we have
$\pi_a=\pi_b+1$. Setting $k=\pi_b$ gives
$\pi_k=b>a=\pi_{k+1}$, so occurrences of $p'$ correspond bijectively
to descents of $\pi$. Therefore,
$\occ_{p'}(\pi)=\operatorname{des}(\pi)$.

Combining these identities with
\eqref{eq:RSK-descent-complement}, we obtain
$\occ_p(\pi)
=
\operatorname{asc}(\pi)
=
\operatorname{des}(\Theta_n(\pi))
=
\occ_{p'}(\Theta_n(\pi)).$
Since $\Theta_n$ is a bijection of $\In$, the two representatives are
equidistributed.

Put
$E_n^{\mathcal I}(q)
=
\sum_{\pi\in\In}q^{\operatorname{des}(\pi)}.$
By the occurrence interpretation above and the bijection $\Theta_n$,
the common distribution polynomial of the representatives in
$\class{046}$ is $E_n^{\mathcal I}(q)$. Thus,
$F^{\mathcal I}_{\class{046}}(x,q)
=
\sum_{n\geq0}E_n^{\mathcal I}(q)x^n.$

We now use the classical generating function for the Eulerian
polynomials on involutions; see \cite[Eq.~(1.1)]{GuoZeng}:
\[
\sum_{n\geq0}
\frac{E_n^{\mathcal I}(q)z^n}{(1-q)^{n+1}}
=
\sum_{r\geq0}
\frac{q^r}
{(1-z)^{r+1}(1-z^2)^{\binom{r+1}{2}}}.
\]
Substituting $z=(1-q)x$ and multiplying both sides by $1-q$, the
left-hand side becomes
$\sum_{n\geq0}E_n^{\mathcal I}(q)x^n
=
F^{\mathcal I}_{\class{046}}(x,q).$
Therefore, the resulting identity is precisely
\eqref{eq:046-generating-function}.
\end{proof}

% \begin{remark}
% The polynomials $E_n^{\mathcal I}(q)$ are the Eulerian polynomials on
% involutions. Guo and Zeng proved that their coefficient sequences are
% unimodal \cite{GuoZeng}, and Wang later proved the stronger property of
% $\gamma$-positivity \cite[Theorem~1.2]{Wang}.
% \end{remark}

\begin{theorem}\label{thm:052}
For every $n\geq0$, all patterns in $\class{052}$ are equidistributed
on $\mathcal I_n$. Moreover,
\[
F^{\mathcal I}_{\class{052}}(x,q)
=
\frac{1+q}{2}\mathcal I(x)
+\frac{1-q}{2}(1+x).
\]
\end{theorem}

\begin{proof}
Consider the representatives
\[
p=
\pattern{scale=0.6}{2}{1/1,2/2}
{0/0,1/0,2/0,0/1,1/1,2/1}
\qquad\text{and}\qquad
p'=
\pattern{scale=0.6}{2}{1/2,2/1}
{0/0,1/0,2/0,0/1,1/1,2/1}.
\]
Fix $n\geq2$. The shading forces both patterns to occur only on the
elements with values $1$ and $2$. Hence, every $\pi$ has at most one
occurrence of either pattern.

The pattern $p$ occurs if and only if
$\pos_\pi(1)<\pos_\pi(2)$, whereas $p'$ occurs if and only if
$\pos_\pi(2)<\pos_\pi(1)$. Since $\pi$ is an involution,
$\pos_\pi(i)=\pi_i$, so these conditions are equivalent to
$\pi_1<\pi_2$ and $\pi_1>\pi_2$, respectively. Thus, $p$ and $p'$ occur precisely according as position $1$ is an ascent or a descent, respectively.

By \eqref{eq:RSK-descent-complement}, $\Theta_n$ exchanges these two
conditions. Since exactly one of them holds for every $\pi\in\In$,
each holds for $I_n/2$ involutions. Therefore,
\[
[x^n]F^{\mathcal I}_{\class{052}}(x,q)
=
\frac{I_n}{2}(1+q)
\qquad(n\geq2).
\]

For $n=0,1$, neither pattern has an occurrence, so the corresponding
terms are $1+x$. Hence,
\[
\begin{aligned}
F^{\mathcal I}_{\class{052}}(x,q)
&=
1+x+\frac{1+q}{2}\sum_{n\geq2}I_nx^n\\
&=
\frac{1+q}{2}\mathcal I(x)
+\frac{1-q}{2}(1+x).
\end{aligned}
\]
\end{proof}

\begin{example}\label{ex:rsk-11}
Let
\[
\pi=(8,4,10,2,5,9,11,1,6,3,7)
=(1\,8)(2\,4)(3\,10)(6\,9)(7\,11)(5).
\]
The RSK correspondence gives $\operatorname{RSK}(\pi)=(T,T)$, where
\[
T=
\begin{array}{|c|c|c|c|}
\hline
1&3&6&7\\ \hline
2&5&11&\multicolumn{1}{c}{}\\ \cline{1-3}
4&9&\multicolumn{2}{c}{}\\ \cline{1-2}
8&10&\multicolumn{2}{c}{}\\ \cline{1-2}
\end{array}
\quad\longmapsto\quad
T^{\mathsf t}=
\begin{array}{|c|c|c|c|}
\hline
1&2&4&8\\ \hline
3&5&9&10\\ \hline
6&11&\multicolumn{2}{c}{}\\ \cline{1-2}
7&\multicolumn{3}{c}{}\\ \cline{1-1}
\end{array}.
\]
Applying inverse RSK to $(T^{\mathsf t},T^{\mathsf t})$ gives
\[
\Theta_{11}(\pi)
=(3,7,1,9,6,5,2,11,4,10,8)
=(1\,3)(2\,7)(4\,9)(5\,6)(8\,11)(10).
\]
Here,
$\Des(\pi)=\{1,3,7,9\},
\Des(\Theta_{11}(\pi))=\{2,4,5,6,8,10\}.$
For $\class{046}$, the consecutive values $a,a+1$ form a
$p$-occurrence in $\pi$ precisely for
$a\in\Des(\Theta_{11}(\pi))$. For the same values of $a$, they form the
$p'$-occurrences in $\Theta_{11}(\pi)$. Conversely, they form the
$p'$-occurrences in $\pi$ and the $p$-occurrences in
$\Theta_{11}(\pi)$ precisely for $a\in\Des(\pi)$. Thus,
$\occ_p(\pi)=\occ_{p'}(\Theta_{11}(\pi))=6,
\occ_{p'}(\pi)=\occ_p(\Theta_{11}(\pi))=4.$

For $\class{052}$, $\pos_\pi(2)<\pos_\pi(1)$ gives the $p'$-occurrence in $\pi$, whereas
$\pos_{\Theta_{11}(\pi)}(1)<\pos_{\Theta_{11}(\pi)}(2)$ gives the $p$-occurrence in $\Theta_{11}(\pi)$. Thus, $\Theta_{11}$
interchanges the two occurrence indicators for $\class{052}$.
\end{example}

\subsection{Prefix-Rank Bijection for Class INV-023}

% ===============

\begin{theorem}\label{thm:023}
For every $n\geq0$, all patterns in $\class{023}$ are equidistributed
on $\mathcal I_n$. Moreover,
\begin{equation}\label{eq:023-generating-function}
F^{\mathcal I}_{\class{023}}(x,q)
=
1+\frac{R(x,q)-1}{q},
\end{equation}
where $R(x,q)$ is defined in
Lemma~\ref{lem:R-recurrence}.
\end{theorem}

\begin{proof}
The case $n=0$ is trivial, so fix $n\geq1$. Consider the representatives
\[
\begin{gathered}
p_0=
\pattern{scale=0.6}{2}{1/1,2/2}{0/0,1/0,0/1,1/1},
\qquad
p_1=
\pattern{scale=0.6}{2}{1/2,2/1}{0/0,1/0,2/0,0/1},
\qquad
p_2=
\pattern{scale=0.6}{2}{1/2,2/1}{0/0,1/0,0/1,1/1},
\qquad
p_3=
\pattern{scale=0.6}{2}{1/2,2/1}{0/0,2/0,0/1,2/1},
\qquad
p_4=
\pattern{scale=0.6}{2}{1/2,2/1}{0/0,0/1,1/1,2/1}.
\end{gathered}
\]
For $p_1, p_2, p_3$, and $p_4$, an occurrence is uniquely determined by a left-to-right minimum $a>1$. Due to the shading, $a$ must be paired with $1$ for $p_1$, with the next left-to-right minimum for $p_2$, with the rightmost element among $1,\ldots,a-1$ for $p_3$, and with $a-1$ for $p_4$. In each case, this partner uniquely exists. 
Consequently, these four patterns have the same number of occurrences. Thus, for every $\pi \in \In$,
\begin{equation}\label{eq:023-lmin}
\occ_{p_1}
=
\occ_{p_2}
=
\occ_{p_3}
=
\occ_{p_4}
=
\lmin-1.
\end{equation}
It remains to construct a bijection $\mathcal{C}$ on $\In$ such that $\occ_{p_0}(\pi) = \occ_{p_2}(\mathcal{C}(\pi))$. For $\pi=\pi_1\cdots\pi_n\in\In$, define the prefix-rank of $\pi_j$ in the prefix $\pi_1\cdots\pi_j$ by
$r_j(\pi)=\#\{k<j:\pi_k<\pi_j\}.$
Thus, $r_j(\pi)=0$ if $\pi_j$ is a left-to-right minimum, and
$r_j(\pi)=1$ if it is the second smallest element of
$\pi_1,\ldots,\pi_j$.
Mark every position $j$ for which $r_j(\pi)\leq1$, assigning it the
mark $r_j(\pi)\in\{0,1\}$, and write the marked positions as
$a_1<a_2<\cdots<a_m$, and note that $a_1 = 1$.
Let $M=\{a_1,\ldots,a_m\}$. Since $\pi$ is an involution, symmetry yields $r_{\pi_j}(\pi) = r_j(\pi)$.
Consequently, the marked values are precisely the marked positions:
$\{\pi_{a_1},\ldots,\pi_{a_m}\}
=
\{a_1,\ldots,a_m\}.$

We now construct $\sigma=\mathcal{C}(\pi)$. Set
$U=M$ and process the marked positions from right to
left. For $t=m,m-1,\ldots,2$, define
\[
\sigma_{a_t}=
\begin{cases}
\min U,
   & r_{a_t}(\pi)=1,\\[1mm]
\text{the second smallest element of }U,
   & r_{a_t}(\pi)=0,
\end{cases}
\]
and delete the chosen value from $U$. Place the remaining value at
position $a_1$, and set $\sigma_j=\pi_j$ at every unmarked position.

For example, take
$\pi=453126789
=(1\,4)(2\,5)(3)(6)(7)(8)(9).$
Then
$(r_1(\pi),\ldots,r_9(\pi))
=(0,1,0,0,1,5,6,7,8),$
so the marked positions are $1,2,3,4,5$. The successive choices are
shown horizontally below:
\[
\begin{array}{c|ccccc}
\text{position processed}
   &5&4&3&2&1\\ \hline
\text{mark}
   &1&0&0&1&0\\
U\text{ before placement}
   &\{1,2,3,4,5\}
   &\{2,3,4,5\}
   &\{2,4,5\}
   &\{2,5\}
   &\{5\}\\
\text{choice}
   &\min&\text{second min.}&\text{second min.}&\min&\text{remaining}\\
\text{value placed}
   &1&3&4&2&5
\end{array}
\]
Thus
$\mathcal{C}(\pi)=524316789$.
% =(1\,5)(2)(3\,4)(6)(7)(8)(9),$
% and
% $(r_1(C(\pi)),\ldots,r_9(C(\pi)))
% =(0,0,1,1,0,5,6,7,8).$
In $\pi$, the two $p_0$-occurrences end at positions $2$ and $5$ (specifically, occupying the pairs $(1,2)$ and $(4,5)$). Correspondingly, the left-to-right minima of $\mathcal{C}(\pi)$ occur at positions $1, 2$, and $5$, yielding the two $p_2$-occurrences on the pairs $(1, 2)$ and $(2, 5)$.

We verify the rank changes by downward induction. Before processing a
position $j$, suppose that the remaining values for $\pi$ and
$\mathcal{C}(\pi)$, written increasingly, are
$x<z_2<\cdots<z_j
\text{ and }
y<z_2<\cdots<z_j.$
This holds initially, when both lists are $1<2<\cdots<n$.
If $j$ is marked and $r_j(\pi)=0$, then $\pi$ removes $x$,
whereas $\mathcal{C}(\pi)$ removes $z_2$. The remaining lists are
$z_2<z_3<\cdots<z_j
\text{ and }
y<z_3<\cdots<z_j.$
If $r_j(\pi)=1$, then $\pi$ removes $z_2$, whereas $\mathcal{C}(\pi)$
removes $y$, and the same relation is preserved. If $j$ is
unmarked, write $r_j(\pi)=k\geq2$. Both procedures then remove
$z_{k+1}$, so its rank remains $k$, and the relation between the
two lists is again preserved.

Therefore,
$r_{a_1}(\mathcal{C}(\pi))=0,
r_{a_t}(\mathcal{C}(\pi))=1-r_{a_t}(\pi)$ $(2\leq t\leq m)$,
and
$r_j(\mathcal{C}(\pi))=r_j(\pi)$ 
$(j\notin M)$.
Thus, complementing the ranks at the marked positions does not affect
the ranks at the unmarked positions. Consequently, the set of marked
positions is preserved:
$M(\mathcal{C}(\pi))=M(\pi).$

It remains to verify that $\mathcal{C}(\pi)$ is an involution.
Indeed, every element strictly southwest of a marked element is itself
marked. Hence, after restricting to $M$ and standardizing, the ranks
of the marked positions remain
$r_{a_1}(\pi),\ldots,r_{a_m}(\pi)$.
This gives an involution reconstructed from the binary rank sequence
$\bigl(r_{a_1}(\pi),\ldots,r_{a_m}(\pi)\bigr)$, whereas the restriction of
$\mathcal{C}(\pi)$ is reconstructed from
$\bigl(0,1-r_{a_2}(\pi),\ldots,1-r_{a_m}(\pi)\bigr).$
We prove by induction on $m$ that the latter restriction is also an
involution. The cases $m\leq2$ are immediate. For $m\geq3$, let
$\varepsilon_t=r_{a_t}(\pi)$. If $\varepsilon_m=0$, then the last
element of the first restriction is $1$. Since this restriction is an
involution, its first element is $m$, and hence $\varepsilon_2=0$.
The complemented reconstruction therefore contains the transposition
$(2\,m)$, in place of $(1\,m)$. If $\varepsilon_m=1$, the first
restriction contains $(2\,m)$ and $\varepsilon_2=1$, while the
complemented reconstruction contains $(1\,m)$. After deleting these
transpositions and standardizing, the remaining rank sequences are
again complementary, apart from their initial zero. The result follows
by induction.
Hence, the marked restriction of $\mathcal{C}(\pi)$ is an involution. Its unmarked restriction 
is unchanged, so $\mathcal{C}(\pi)\in\In$.

The bijection preserves the marked positions, leaves the unmarked
ranks unchanged, and complements every marked rank except the first.
Applying it twice therefore restores the entire rank sequence of
$\pi$. Since a permutation is uniquely determined by its rank
sequence through right-to-left reconstruction, we have
$\mathcal{C}^2(\pi)=\pi.$
Thus, $\mathcal{C}$ is a self-inverse bijection on~$\In$.

Finally, a $p_0$-occurrence ending at position $j$ exists precisely when
$r_j(\pi)=1$; in that case, its first selected element is the unique
earlier element smaller than $\pi_j$. 
\[
\begin{aligned}
\occ_{p_0}(\pi)
&=\#\{j:r_j(\pi)=1\}\\
&=\#\{j:r_j(\mathcal{C}(\pi))=0\}-1\\
&=\lmin(\mathcal{C}(\pi))-1\\
&=\occ_{p_2}(\mathcal{C}(\pi)),
\end{aligned}
\]
where the last equality follows from \eqref{eq:023-lmin}. Thus,
$p_0$ and $p_2$ are equidistributed. Together with
\eqref{eq:023-lmin}, this proves that all five representatives are
equidistributed.

% =========
It remains to determine the common distribution. By
\eqref{eq:023-lmin} and the bijection $\mathcal{C}$, for every
$p\in\class{023}$ and $n\geq1$,
$
\sum_{\pi\in\mathcal I_n}q^{\occ_p(\pi)}
=
\sum_{\pi\in\mathcal I_n}q^{\lmin(\pi)-1}
=
\frac{R_n(q)}{q}.
$

The quotient is a polynomial in $q$, since every nonempty
permutation has at least one left-to-right minimum. Summing over $n\geq1$ and including the empty involution gives \eqref{eq:023-generating-function}.
\end{proof}

\begin{remark}
The proof gives
$\occ_{p_0}(\pi)
=
\lmin(\mathcal C(\pi))-1.$
Let $p_{004}$ be either of the following two patterns in $\class{004}$, displayed in
the proof of Theorem~\ref{thm:004}:
$\pattern{scale=0.6}{2}{1/1,2/2}{0/0,1/0}$
or
$\pattern{scale=0.6}{2}{1/1,2/2}{1/0,1/1}.$
Then
$\occ_{p_{004}}(\mathcal C(\pi))
=
n-\lmin(\mathcal C(\pi)).$
Hence
$\occ_{p_0}(\pi)
+
\occ_{p_{004}}(\mathcal C(\pi))
=
n-1.$
Since $\mathcal C$ is a bijection of $\mathcal I_n$, and all patterns
within the respective classes are equidistributed, for every
$n\geq1$,
\[
[x^n]F^{\mathcal I}_{\class{023}}(x,q)
=
q^{n-1}
[x^n]F^{\mathcal I}_{\class{004}}(x,q^{-1}).
\]
\end{remark}

\subsection{Right-to-Left Prefix-Swap Bijection for Class INV-125}

\begin{theorem}\label{thm:125}
For every $n\geq0$, all patterns in $\class{125}$ are equidistributed
on $\mathcal I_n$. Moreover,
\begin{equation}\label{eq:125-generating-function}
F^{\mathcal I}_{\class{125}}(x,q)
=
\frac{
\mathcal I(x)-q(q-1)\bigl(\mathcal P(x^2)-1\bigr)
}{
1-(q^2-1)\bigl(\mathcal P(x^2)-1\bigr)
}.
\end{equation}
\end{theorem}

\begin{proof}
For the representatives
\[
p=
\pattern{scale=0.6}{2}{1/1,2/2}
 {0/0,1/0,0/1,1/1,2/2}
\qquad\text{and}\qquad
p'=
\pattern{scale=0.6}{2}{1/2,2/1}
 {0/0,1/0,0/1,1/1,2/2},
\]
we construct an involutive bijection
$\Phi_n:\mathcal I_n\to\mathcal I_n$.

Fix $\pi\in\mathcal I_n$, and let
$\eta_1>\cdots>\eta_m$ be the second selected positions of all occurrences
of $p$ or $p'$. Process these positions from right to left. At a
marked position $\eta_i$, interchange the two smallest values in the
prefix ending at $\eta_i$; let $\Phi_n(\pi)$ be the permutation obtained
after all such swaps.

For example,
$57861423
\xrightarrow{(1\,4)}
57864123
\xrightarrow{(4\,5)}
47865123
\xrightarrow{(4\,6)}
67845123.$
Thus $\Phi_8(57861423)=67845123$, and the two occurrences counts
change from $(2,1)$ to $(1,2)$.

Suppose that an occurrence ending at $\eta$ uses positions $u<\eta$ and
values $a<b$. The shading shows that $a$ and $b$ are precisely the
two smallest values in the prefix ending at $\eta$. In particular, each marked position determines a unique occurrence. Hence the swap at
$\eta$ exchanges a $p$-occurrence with a $p'$-occurrence. The same
shading conditions also show that the set of marked positions is
unchanged by each swap. Moreover, the occurrence at any other marked position remains an occurrence of the same pattern, $p$ or $p'$.

The swaps may be performed in any order. Indeed, if two of the swaps
are $(a\,b)$ and $(c\,d)$, then performing $(c\,d)$ first replaces
$(a\,b)$ by its conjugate under $(c\,d)$, and
$\bigl((c\,d)(a\,b)(c\,d)\bigr)(c\,d)
=(c\,d)(a\,b).$
Thus the final permutation is independent of the order.
Since the marked positions are preserved and each swap
is involutive, $\Phi_n^2=\mathrm{id}$.

Finally, transposition of the plot pairs an occurrence with positions
$u<\eta$ and values $a<b$ with the corresponding occurrence in
$\pi^{-1}$. The associated swaps satisfy
$\bigl((a\,b)\pi\bigr)^{-1}
=\pi^{-1}(a\,b)
=(u\,\eta)\pi^{-1}.$
Hence $\Phi_n(\pi^{-1})=\Phi_n(\pi)^{-1}$, so $\Phi_n$ preserves
$\mathcal I_n$. Therefore
$\occ_p(\pi)=\occ_{p'}(\Phi_n(\pi)),$
and the two representatives are equidistributed on involutions.

Let $A(x)$ be the
generating function for involutions avoiding $p$.
Suppose that $\pi$ contains $p$, and choose an occurrence with the smallest second selected position. Let this
occurrence be formed by the elements
$\pi_u=a$ and $\pi_{k+1}=b$, where $u<k+1$ and $a<b$.
The shading forces the other $k-1$ elements preceding position
$k+1$ to be precisely the values larger than $b$. Hence
$n-b=k-1$, so $b=n+1-k$.
Reflection in the main diagonal yields an occurrence with second selected position $b$, since $k+1$ is the smallest second selected position, it follows that $k+1 \le b$, and hence $k \le \lfloor n/2 \rfloor$.

For a fixed $k$, without requiring $k+1$ to be the smallest second
selected position, we first choose $u\in[k]$ in $k$ ways, and then
pair the remaining $k-1$ positions in $[k] \setminus \{u\}$
 with
$b+1,\ldots,n$ in $(k-1)!$ ways.
If $n>2k$, the forced matching
part involves $2k$ elements, with $k+1$ paired with $b$. If
$n=2k$, then $b=k+1$; in this case $k+1$ is a fixed point, and the
shading conditions together with $\pi=\pi^{-1}$ force $u$ to be a
fixed point as well. Thus in either case there are
$k(k-1)!=k!$
possible matchings. Hence all such matchings are counted by
$\mathcal P(x^2)=\sum_{k\geq0}k!x^{2k}.$

Let $B(x)$ count the possible first blocks determined by the
leftmost occurrence. Every nonempty matching has a unique first
block: indeed, its smallest second selected position determines the
leftmost occurrence and hence the corresponding block uniquely.
Deleting this block and standardizing leaves an arbitrary matching.
Hence
$$
\mathcal P(x^2)=1+B(x)\mathcal P(x^2),
\qquad
B(x)=1-\frac{1}{\mathcal P(x^2)}.
$$

Every involution either avoids $p$, or has a unique first block
determined by its leftmost occurrence. The remaining elements are
paired only among themselves, so deleting the first block and
standardizing leaves an involution. Conversely, a first block
together with an arbitrary involution uniquely reconstructs $\pi$.
Therefore
$$
\mathcal I(x)
=
A(x)+B(x)\mathcal I(x),
\qquad
A(x)=\frac{\mathcal I(x)}{\mathcal P(x^2)}.
$$

The same decomposition also records the number of occurrences.
Reflection in the main diagonal sends an occurrence ending at $\eta$ to
one ending at $n-\eta+2$. Since $k+1$ is the smallest second selected
position, every occurrence therefore ends at some
$k+1\le \eta\le b.$
If the remainder is nonempty, then $k+1<b$. The occurrences ending
at $k+1$ and $b$ are distinct. For $k+1<\eta<b$, deleting the first
block and standardizing preserves the shading conditions, so these
occurrences are precisely the occurrences in the remainder. Thus the
first block contributes exactly two occurrences.

If the remainder is empty, then $b=k+1$. The chosen occurrence is
fixed by reflection, and the preceding bounds show that it is the
only occurrence. Hence
$$
F^{\mathcal I}_{\class{125}}(x,q)
=
A(x)+qB(x)
+q^2B(x)\bigl(F^{\mathcal I}_{\class{125}}(x,q)-1\bigr).
$$
Substituting the expressions for $A(x)$ and $B(x)$ gives
\eqref{eq:125-generating-function}, completing the proof.
\end{proof}

\subsection{Adjacent-Transposition Bijection for Class INV-102}

\begin{theorem}\label{thm:102}
For every $n\geq0$, all patterns in $\class{102}$ are equidistributed
on $\mathcal I_n$. Moreover,
\begin{equation}\label{eq:102-generating-function}
F^{\mathcal I}_{\class{102}}(x,q)
=
\mathcal I(x)+(q-1)\sum_{n\geq2}a_nx^n,
\end{equation}
where $a_n$ denotes the number of involutions in $\mathcal I_n$
containing an occurrence. We have $a_2=1$, $a_3=2$, and, for
$n\geq4$,
\[
a_n
=
I_{n-2}
+\sum_{j=0}^{n-3}I_jI_{n-3-j}
+\sum_{k=0}^{\lfloor(n-4)/2\rfloor}
(-1)^k k!\binom{n-2}{2k+2}I_{n-4-2k}.
\]
\end{theorem}

\begin{proof}
For the representatives
\[
p=
\pattern{scale=0.6}{2}{1/1,2/2}{0/0,1/0,2/0,1/1,2/1,1/2}
\qquad\text{and}\qquad
p'=
\pattern{scale=0.6}{2}{1/2,2/1}{0/0,1/0,2/0,1/1,2/1,1/2},
\]
the shading immediately shows that both patterns occur at most once. Let $t=\pi_1$ be the position of the value $1$. 
The pattern $p$ occurs when $t<n$, the second selected value is $b=\pi_{t+1}$, and all elements to the right of $t+1$ are strictly greater than $b$. 
For $p'$, an occurrence exists when $t>1$, the larger selected value is
$b=\pi_{t-1}$, and all elements to the right of $t$ are greater than~$b$.

We construct a bijection $\Phi$ mapping $p$-occurrences to $p'$-occurrences on $\In$:
\begin{itemize}
    \item \textbf{Case 1 ($t=1$):} Since every element to the right of position $t+1$ is greater than $b$, we must have $b=2$. Define $\Phi(\pi)=(1\,2)\pi$. As $\pi$ fixes $1$ and $2$, replacing these fixed points with the transposition $(1\,2)$ preserves involutivity and sends $1$ to position $2$, creating a $p'$-occurrence with $t=2$.

    \item \textbf{Case 2 ($t>1$):} Define $\Phi(\pi)=s_t\pi s_t$, where $s_t=(t\,\,t+1)$. Then conjugation preserves involutions, $\Phi(\pi)$ is again an involution and $\Phi(\pi)_{t+1}=1$. The suffix condition implies
$[b]\subseteq\{\pi_1,\ldots,\pi_{t+1}\}$, so $b\leq t+1$.
Moreover, $b\neq t$, since $\pi_{t+1}=t$ would imply
$\pi_t=t+1$, contradicting $\pi_t=1$. Hence either $b<t$ or
$b=t+1$. If $b<t$, then $\Phi(\pi)_t=b$; moreover,
involutivity gives $\pi_1=t$ and $\pi_b=t+1$, so all elements to
the right of position $t+1$ remain greater than $b$. If
$b=t+1$, then $\Phi(\pi)_t=t$, while all elements to the right of
position $t+1$ are greater than $t+1$. Thus $\Phi(\pi)$ contains
a $p'$-occurrence.

\end{itemize}

For an involution $\sigma$ containing $p'$, let $t=\pos_\sigma(1)$. If $t=2$, then $\Phi^{-1}(\sigma)=(1\,2)\sigma$. If $t>2$, then $\Phi^{-1}(\sigma)=(t-1\,\,t)\sigma(t-1\,\,t)$. These two cases reverse Cases $1$ and $2$, respectively. Hence $\Phi$ is a bijection between the involutions containing $p$ and those containing $p'$.

\begin{example}
For Case~1, let
$\pi=1\,2\,6\,5\,4\,3=(1)(2)(3\,6)(4\,5)$. Here $t=1$, and
the elements $1,2$ in the first two positions form a $p$-occurrence. We have
\[
\Phi(\pi)=(1\,2)\pi=(1\,2)(3\,6)(4\,5)
          =2\,1\,6\,5\,4\,3,
\]
where the elements $2,1$ in the first two positions form the
corresponding $p'$-occurrence.

For Case~2, let
$\pi=3\,4\,1\,2\,5\,6=(1\,3)(2\,4)(5)(6)$. Here $t=3$, and
the elements $1,2$ in positions $3,4$ form a $p$-occurrence. We have
\[
\Phi(\pi)=s_3\pi s_3
=(3\,4)(1\,3)(2\,4)(3\,4)
=(1\,4)(2\,3)(5)(6)
=4\,3\,2\,1\,5\,6.
\]
The elements $2,1$ in positions $3,4$ now form a $p'$-occurrence.
\end{example}

Let $a_n$ be the number of involutions in $\mathcal I_n$ containing $p$. Since $p$
occurs at most once,
$\sum_{\pi\in\mathcal I_n}q^{\occ_p(\pi)}
=
(I_n-a_n)+qa_n.$
We now count $a_n$. Let $t=\pi_1=\pos_\pi(1)$ and $b=\pi_{t+1}$. Suppose first that $t=1$. Since every element to the right of $b=\pi_2$ is larger than $b$, we must have $b=2$. Thus $1$ and $2$ are fixed, while the remaining $n-2$ elements can form an arbitrary involution, giving $I_{n-2}$ possibilities.
Next suppose that $t>1$ and $b=t+1$. Then $(1\,t)$ is a transposition and $t+1$ is fixed. The elements in positions $\{2,\ldots,t-1\}$ and $\{t+2,\ldots,n\}$ can form arbitrary involutions independently. Hence this case contributes
$\sum_{t=2}^{n-1} I_{t-2} I_{n-t-1}.$
% $= \sum_{j=0}^{n-3} I_j I_{n-3-j}. $

Finally, suppose that $2\le b\le t-1$. The transpositions
$(1\,t)$ and $(b\,\,t+1)$ are forced. The remaining $n-4$ labels form the following three disjoint
intervals:
\[
L = \{2, \ldots, b-1\}, \quad M = \{b+1, \ldots, t-1\}, \quad R = \{t+2, \ldots, n\}.
\]
Let their respective lengths be $\ell, m,$ and $r$, which satisfy $\ell+m+r = n-4$. 

The occurrence condition imposes the restriction that no element of
$L$ can be paired with an element of $R$; pairings involving $M$ are
otherwise allowed. 
Hence, by inclusion--exclusion, for fixed
$\ell,m,r$ these involutions are counted by
\[
\sum_{k=0}^{\min(\ell,r)}
(-1)^k
\binom{\ell}{k}\binom{r}{k}k!\,
I_{n-4-2k}.
\]
Indeed, for $0\leq k\leq\min(\ell,r)$, choose $k$ elements from each
of $L$ and $R$ and pair them in $k!$ ways; the remaining
$n-4-2k$ elements can form an arbitrary involution.
Using
\[
\sum_{\substack{\ell+m+r=n-4\\ \ell,r\geq k}}
\binom{\ell}{k}\binom{r}{k}
=
\binom{n-2}{2k+2},
\]
we obtain
\[
\sum_{k=0}^{\lfloor(n-4)/2\rfloor}
(-1)^k k!\binom{n-2}{2k+2}I_{n-4-2k}
\]
possibilities in this case.

Combining the three cases gives the stated formula for $a_n$, and hence
\[
F^{\mathcal I}_{\class{102}}(x,q)
=
\mathcal I(x)+(q-1)\sum_{n\geq2}a_nx^n.
\qedhere
\]
\end{proof}

\section{Equidistribution via a marked-occurrence bijection}
\label{sec:marked-occurrence-bijection}

In this section, we prove the remaining equidistribution by means of
marked occurrences. In the involution setting, a marked occurrence
must be considered together with its reflection in the main diagonal.
We therefore adapt the marked-occurrence bijection of Zhang and
Zhao~\cite{ZhangZhao} so that it preserves involutions.

\begin{theorem}\label{thm:099}
For every $n\geq0$, all patterns in $\class{099}$ are equidistributed
on $\mathcal I_n$.
\end{theorem}

\begin{proof}
  We consider
\[
P=
\pattern{scale=0.6}{2}{1/1,2/2}{0/0,1/1,2/1,1/2}
\qquad\text{and}\qquad
Q=
\pattern{scale=0.6}{2}{1/2,2/1}{0/0,1/1,2/1,1/2}.
\]
Their common shading shows
that $P^{-1}=P$ and $Q^{-1}=Q$.
For a permutation $\tau$, put $\rho_a=\pos_\tau(a)$. We identify an
occurrence with the unordered edge $\{a,b\}$ on its selected values,
where $a<b$. The classical part gives $\rho_a<\rho_b$ for a $P$-edge and
$\rho_b<\rho_a$ for a $Q$-edge. For $k\geq0$, define
\[
\begin{aligned}
\mathcal X_{n,k}
&=\{(\pi,M):\pi\in\mathcal I_n,\,
M\subseteq\Occ_P(\pi),\,|M|=k\},\\
\mathcal Y_{n,k}
&=\{(\sigma,N):\sigma\in\mathcal I_n,\,
N\subseteq\Occ_Q(\sigma),\,|N|=k\}.
\end{aligned}
\]
We construct mutually inverse maps between these two sets.

We first determine the possible graphs of marked occurrences. For
$P$-edges, the shading gives two restrictions.
\begin{itemize}
\item \emph{A vertex cannot be the smaller endpoint of two marked
$P$-edges.}
Suppose that $\{a,b\}$ and $\{a,c\}$ are marked $P$-edges with
$a<b<c$. In the occurrence $\{a,c\}$, the intermediate value $b$
cannot occur to the right of the selected element of value $a$. Hence
$\rho_b<\rho_a$, contrary to the orientation $\rho_a<\rho_b$ of the $P$-edge
$\{a,b\}$.

\item \emph{A vertex cannot be the larger endpoint of one marked
$P$-edge and the smaller endpoint of another.}
If $\{a,b\}$ and $\{b,c\}$ were both marked $P$-edges with $a<b<c$,
then in the occurrence $\{b,c\}$ the value $a<b$ could not lie to the
left of the selected element of value $b$. Thus $\rho_a>\rho_b$, again contradicting the
orientation of $\{a,b\}$.
\end{itemize}
Any two adjacent marked P-edges must therefore share their larger‑value endpoint. By connectedness, every nontrivial component has a unique largest‑value center, so the marked‑occurrence graph is a star forest.

For $Q$-edges, the two corresponding restrictions are as follows.
\begin{itemize}
\item \emph{Every vertex has at most one smaller neighbour.}
Suppose that $a<c<b$ and that both $\{a,b\}$ and $\{c,b\}$ are marked.
In the occurrence $\{a,b\}$, the intermediate value $c$ must lie to
the left of the selected element of value $b$, thus $\rho_c<\rho_b$, contrary
to the orientation $\rho_b<\rho_c$ of the edge $\{c,b\}$.

\item \emph{Every vertex has at most one larger neighbour.}
Suppose that $a<b<c$ and that both $\{a,b\}$ and $\{a,c\}$ are marked.
The occurrence $\{a,c\}$ forces $\rho_b<\rho_c<\rho_a$. But then, in the
$Q$-edge $\{a,b\}$, the element of value $c$ lies between the two
selected positions and above the larger selected value $b$, so it lies in the shaded cell $(1,2)$, a contradiction.
\end{itemize}
The strict decrease of positions along an oriented path excludes cycles, so every nontrivial component is a path $a_1-a_2-\cdots-a_m$ with $a_1<a_2<\cdots<a_m$. Each edge is a 21-pattern occurrence, which forces $\rho_{a_1}>\rho_{a_2}>\cdots>\rho_{a_m}$. Hence the marked‑occurrence graph is a disjoint union of such decreasing‑position Q-paths.

We next close each marked set under reflection in the main diagonal.
For $\pi\in\mathcal I_n$, this reflection sends an occurrence edge $e=\{a,b\}$ to
$
\theta_\pi(e)=\{\pi(a),\pi(b)\}.
$
Since $P^{-1}=P$ and $Q^{-1}=Q$, the map $\theta_\pi$ preserves
occurrences of either pattern. Moreover,
$\theta_\pi^2=\mathrm{id}$ since $\pi^2=\mathrm{id}$.

Although the plot of an involution is symmetric, an arbitrary marked
set $M$ need not contain the reflected copy of each of its edges. We
therefore define its symmetric closure by
$\cl_\pi(M)=M\cup\theta_\pi(M)$. This is the smallest
$\theta_\pi$-stable marked set containing $M$, and it forces every
marked feature to occur together with its diagonal reflection. This
pairing is what later makes the component cycles occur in reflected
inverse pairs. In particular, if
$(\pi,M)\in\mathcal X_{n,k}$ and $S=\cl_\pi(M)$, then
$S\subseteq\Occ_P(\pi)$ and $S$ is a $\theta_\pi$-stable star forest.

We now define the forward bijection $\Phi_{n,k}$.
Take $(\pi,M)\in\mathcal X_{n,k}$ and set $S=\cl_\pi(M)$. We perform
two separate operations: the first changes the underlying permutation
$\pi$, and the second changes the marked set $M$.
\begin{enumerate}
\item \emph{Changing the permutation $\pi$.}
For every nontrivial star component
$C=\{a_1<a_2<\cdots<a_m\}$, whose edges are
$\{a_i,a_m\}$ for $1\leq i<m$, apply the following cyclic shift to the
values:
\[
a_1\mapsto a_2,\quad
a_2\mapsto a_3,\quad
\ldots,\quad
a_{m-1}\mapsto a_m,\quad
a_m\mapsto a_1.
\]
Different connected components have no common vertices. We therefore
perform these cyclic shifts independently on all components of $S$.
Let $\gamma_S$ denote the resulting simultaneous value replacement;
it leaves every value outside the nontrivial components unchanged.
The target permutation is
$\sigma=\gamma_S\pi$. Equivalently,
$(\gamma_S\pi)(i)=\gamma_S(\pi(i))$, so the elements of the one-line
notation of $\pi$ are replaced according to the displayed rules.

\item \emph{Changing the marked set $M$.}
On each of the same components, replace every star edge
$\{a_i,a_m\}$ with the path edge $\{a_i,a_{i+1}\}$. This rule is
defined on the full closure $S$ and gives a bijection
$\lambda_S:S\longrightarrow\lambda_S(S)$. Only the images of the
edges originally belonging to $M$ are marked in the target object; we
therefore set $N=\lambda_S(M)$.
\end{enumerate}

This componentwise transformation is adapted from the marked-occurrence
bijection of Zhang and Zhao~\cite[Section~2]{ZhangZhao}.
More precisely, position reversal does not send our $P$ directly to
our $Q$. It sends our $P$ to the target representative used by Zhang
and Zhao, and it sends our $Q$ to their source representative.
Therefore, starting from our $P$, we first reverse positions to reach
their target, apply the inverse of their bijection to reach their
source, and then reverse positions once more to reach our $Q$. After
this change of coordinates, their inverse bijection is exactly the
componentwise transformation defined above:
$(\tau,S)\longmapsto(\gamma_S\tau,\lambda_S(S)).$
Thus this is not merely analogous to our rule; it is the same rule
written in our coordinates. Specifically, the left, middle, and right
region checks in that section show for each component that every image
edge is a $Q$-occurrence, and the subsequent two-component check shows
that performing the replacements on disjoint components does not
interfere with those occurrences. Therefore,
$\lambda_S(S)\subseteq\Occ_Q(\gamma_S\tau)$. Applying this with
$\tau=\pi$ gives
$N\subseteq\lambda_S(S)\subseteq\Occ_Q(\sigma)$. Since $\lambda_S$ is
injective, $|N|=|M|=k$. When $k=0$, all the products and marked sets
are empty, and the bijection is the identity.

It remains to prove that $\sigma$ is an involution and that the
symmetric closure can be recovered from $(\sigma,N)$. Let
$C=\{a_1<\cdots<a_m\}$ be a star component of $S$, and write
$\pi(C)=\{b_1<\cdots<b_m\}$. Because $S$ is
$\theta_\pi$-stable, $\theta_\pi$ is an automorphism of its occurrence
graph, so $\pi(C)$ is again a star component of $S$.

Every $\{a_i,a_m\}$ is a $P$-edge. Hence $\rho_{a_i}<\rho_{a_m}$ for
$i<m$. If $i<j<m$, then in the occurrence $\{a_i,a_m\}$ the
intermediate value $a_j$ cannot lie to the right of the element of value
$a_i$, because it would lie in one of the shaded cells $(1,1)$ or
$(2,1)$. Therefore $\rho_{a_j}<\rho_{a_i}$. Since $\pi$ is an involution,
$\rho_a=\pi(a)$, and consequently
$\pi(a_{m-1})<\cdots<\pi(a_2)<\pi(a_1)<\pi(a_m).$
It follows that $\pi(a_i)=b_{m-i}$ for $1\leq i<m$, while
$\pi(a_m)=b_m$. Geometrically, the relabelling
$a\mapsto\pi(a)$ is induced by reflection in the main diagonal. The
displayed order shows that this reflection reverses the leaf order of
the star while sending its largest center to the largest center of the
image star. Consequently, the forward cyclic order on $C$ is seen as
the reverse cyclic order on $\pi(C)$. Algebraically, conjugating
$(a_1\,a_2\,\cdots\,a_m)$ by $\pi$ therefore gives the inverse
component cycle on $\pi(C)$. Since $\pi$ permutes the components of
$S$, taking the product over all components gives
$\pi\gamma_S\pi=\gamma_S^{-1}.$
Thus
$\sigma^2=(\gamma_S\pi)^2
=\gamma_S(\pi\gamma_S\pi)=\mathrm{id}$, and
$\sigma\in\mathcal I_n$.

We next show that the edge replacement respects the diagonal pairing
used in the symmetric closure. Concretely, reflecting an edge first
and then replacing it must give the same target edge as replacing it
first and then reflecting it in the plot of $\sigma$. This fact is
needed because the inverse map receives only $(\sigma,N)$ and must
reconstruct the full transformed forest $\lambda_S(S)$ from the
symmetric closure of $N$. For $e_i=\{a_i,a_m\}$, the component order
gives
$\theta_\pi(e_i)=\{b_{m-i},b_m\}$. If $i<m-1$, then
\[
\theta_\sigma(\lambda_S(e_i))
=\{\sigma(a_i),\sigma(a_{i+1})\}
=\{b_{m-i},b_{m-i+1}\}
=\lambda_S(\theta_\pi(e_i)).
\]
For $i=m-1$, the cyclic shift on $\pi(C)$ sends $b_m$ to $b_1$ and
$b_1$ to $b_2$, so both sides are the unordered edge $\{b_1,b_2\}$.
Hence
$\theta_\sigma(\lambda_S(e))
=\lambda_S(\theta_\pi(e))$ for every $e\in S$. Using
$S=M\cup\theta_\pi(M)$ and $N=\lambda_S(M)$, we obtain
\[
\cl_\sigma(N)
=N\cup\theta_\sigma(N)
=\lambda_S(M)\cup\lambda_S(\theta_\pi(M))
=\lambda_S(S).
\]
Thus the inverse map can recover the full path forest used by the
forward bijection from $(\sigma,N)$ alone. Therefore the
bijection defines a map
$\Phi_{n,k}:\mathcal X_{n,k}\longrightarrow\mathcal Y_{n,k}$.

We now construct the inverse map $\Psi_{n,k}$. Given $(\sigma,N)\in\mathcal Y_{n,k}$, let $T=\cl_\sigma(N)$. Since $\theta_\sigma$ preserves $Q$-occurrences, $T\subseteq\Occ_Q(\sigma)$, and every nontrivial component of $T$ is a decreasing $Q$-path. For each component
$D=\{a_1<\cdots<a_m\}$, apply the inverse cyclic shift
\[
a_1\mapsto a_m,\quad
a_m\mapsto a_{m-1},\quad
\ldots,\quad
a_3\mapsto a_2,\quad
a_2\mapsto a_1,
\]
 and replace every path edge $\{a_i,a_{i+1}\}$ by the star edge $\{a_i,a_m\}$. Let $\delta_T$ be the product of these cycles over the disjoint components of $T$, and denote the resulting edge bijection by $\mu_T$. Define $\pi=\delta_T\sigma$, $M=\mu_T(N)$, and $\Psi_{n,k}(\sigma,N)=(\pi,M)$. The inverse componentwise transformation gives $\mu_T(T)\subseteq\Occ_P(\pi)$. Moreover, the calculations for the forward map reverse to give $\sigma\delta_T\sigma=\delta_T^{-1}$ and $\theta_\pi\circ\mu_T=\mu_T\circ\theta_\sigma$. Hence $\pi^2=\mathrm{id}$ and $\cl_\pi(M)=\mu_T(T)$. Therefore, $\pi\in\In$, $M\subseteq\Occ_P(\pi)$, and $|M|=|N|=k$, so $\Psi_{n,k}$ is well defined.

If $(\sigma,N)=\Phi_{n,k}(\pi,M)$ and $S=\cl_\pi(M)$, then $T=\lambda_S(S)$, $\delta_T=\gamma_S^{-1}$, and $\mu_T=\lambda_S^{-1}$. Thus $\Psi_{n,k}(\sigma,N)=(\pi,M)$. Conversely, starting from $T$, we have $\lambda_{\mu_T(T)}=\mu_T^{-1}$ and $\gamma_{\mu_T(T)}=\delta_T^{-1}$, which gives $\Phi_{n,k}\circ\Psi_{n,k}=\mathrm{id}$. Hence $\Phi_{n,k}$ and $\Psi_{n,k}$ are mutually inverse bijections.

For a fixed involution $\pi$, there are
$\binom{\occ_P(\pi)}{k}$ ways to choose $k$ distinct
$P$-occurrences. Hence the bijection gives, for every $k\geq0$,
\[
\sum_{\pi\in\mathcal I_n}\binom{\occ_P(\pi)}{k}
=
\sum_{\sigma\in\mathcal I_n}\binom{\occ_Q(\sigma)}{k}.
\]
Multiplying by $(q-1)^k$, summing over $k\geq0$, and using
$q^m=\sum_{k\geq0}\binom{m}{k}(q-1)^k$, we obtain
\[
\sum_{\pi\in\mathcal I_n}q^{\occ_P(\pi)}
=
\sum_{\sigma\in\mathcal I_n}q^{\occ_Q(\sigma)}.
\]
Thus $P$ and $Q$ have the same distribution on $\mathcal I_n$.
Together with Lemma~\ref{lem:involution-symmetries}, this proves that
all patterns in $\class{099}$ are equidistributed.
\end{proof}

\begin{example}\label{ex:099-map}
Let $P$ and $Q$ be the representatives used in the proof. Take
$\pi=5462137=(1\,5)(2\,4)(3\,6)(7)$ and mark the two
$P$-occurrences $M=\{\{1,3\},\{2,3\}\}$. Diagonal reflection gives
$\theta_\pi(\{1,3\})=\{5,6\}$ and
$\theta_\pi(\{2,3\})=\{4,6\}$. Hence
$S=\cl_\pi(M)=\{\{1,3\},\{2,3\},\{4,6\},\{5,6\}\}$, the disjoint
union of the star on $1<2<3$ centered at $3$ and the star on
$4<5<6$ centered at $6$.

On the first component, cyclically shift the values by
$1\mapsto2$, $2\mapsto3$, and $3\mapsto1$; on the second, use
$4\mapsto5$, $5\mapsto6$, and $6\mapsto4$. The value $7$
is unchanged. Thus $\gamma_S=(1\,2\,3)(4\,5\,6)$, and this operation
changes the permutation $\pi$ to
$\sigma=\gamma_S\pi=6543217=(1\,6)(2\,5)(3\,4)(7)$. Independently,
$\lambda_S$ changes the occurrence edges: it replaces the two stars
in $S$ by the paths with edges $\{1,2\},\{2,3\}$ and
$\{4,5\},\{5,6\}$, respectively. Since only the edges of the first
star belong to $M$, their image is
$N=\lambda_S(M)=\{\{1,2\},\{2,3\}\}$. Therefore the forward map is
\[
\Phi_{7,2}\bigl(5462137,\{\{1,3\},\{2,3\}\}\bigr)
=
\bigl(6543217,\{\{1,2\},\{2,3\}\}\bigr).
\]

The four edges $\{1,2\},\{2,3\},\{4,5\},\{5,6\}$ are
$Q$-occurrences of $\sigma$. Moreover,
$\theta_\sigma(\{1,2\})=\{5,6\}$ and
$\theta_\sigma(\{2,3\})=\{4,5\}$. It follows that
$T=\cl_\sigma(N)=\{\{1,2\},\{2,3\},\{4,5\},\{5,6\}\}$, the
disjoint union of the decreasing $Q$-paths on $1<2<3$ and
$4<5<6$.

For the inverse map,
$\delta_T=(1\,3\,2)(4\,6\,5)=\gamma_S^{-1}$, and $\mu_T$ sends
the marked path edges $\{1,2\},\{2,3\}$ back to the star
edges $\{1,3\},\{2,3\}$. Hence
\[
\Psi_{7,2}\bigl(6543217,\{\{1,2\},\{2,3\}\}\bigr)
=
\bigl(5462137,\{\{1,3\},\{2,3\}\}\bigr),
\]
which recovers the original marked involution.
\end{example}

% \begin{proposition}\label{prop:099-distribution}
% Let

% $$
% A_n(q)
% =
% \sum_{\pi\in\mathcal I_n}q^{\occ_P(\pi)}
% =
% \sum_{\pi\in\mathcal I_n}q^{\occ_Q(\pi)}.
% $$

% Then $A_0(q)=A_1(q)=1$. For $n\geq2$, put $m=n-2$.
% For $\sigma\in\mathcal I_m$ and $1\leq t\leq m+1$, define

% $$
% e_t(\sigma)
% =
% \#\{(i,j)\in\Occ_P(\sigma):
% i<t,\ \sigma_i<t\},
% $$

% and

% $$
% u_t(\sigma)
% =
% \begin{cases}
% 1,
% & t\leq m\text{ and }\sigma_t
%   \text{ is a right-to-left minimum of }\sigma,\\
% 0,
% & t=m+1.
% \end{cases}
% $$

% Then
% \begin{equation}\label{eq:099-insertion}
% A_n(q)
% =
% J_{n-1}+(q-1)J_{n-2}
% +
% \sum_{\sigma\in\mathcal I_{n-2}}
% \sum_{t=1}^{n-1}
% q^{e_t(\sigma)+2u_t(\sigma)}.
% \end{equation}
% \end{proposition}

\section{Conclusion}
\label{conclude}

The main result of this paper is a complete classification of the distribution of length-two mesh patterns on involutions. Among the $1024$ such patterns, we obtain exactly $287$ equidistribution classes, of which $262$ are trivial and $25$ are nontrivial. In addition, we derive explicit distribution formulas for $24$ of the nontrivial classes. Obtaining explicit formulas for the distributions of the remaining classes is left as an open direction for future research. Some of these cases admit simple enumerations; for example, this occurs when all cells of the mesh pattern are shaded. However, we expect many of the other cases to exhibit substantially richer structure, and obtaining closed forms or effective recurrences for their distributions may require new combinatorial arguments.

\section*{Acknowledgments}
Philip B.~Zhang was supported by the National Natural Science
	Foundation of China (Grant No.~12171362) and Natural Science Foundation of
	Tianjin Municipality (Grant No.~25JCYBJC00430).

\appendix

\section*{Appendix}
\label{app:trivial-classes}

Table~\ref{tab:involution-trivial} lists the $262$ classes that consist of a
single $\langle i,rc\rangle$-orbit. Their equivalences therefore follow from
inverse and reverse--complement symmetries. The $25$ nontrivial classes,
each containing more than one such orbit, are omitted.

\begingroup
\renewcommand{\arraystretch}{1.4}
\setlength{\tabcolsep}{1pt}
\setlength{\LTleft}{-9mm plus 1fill}
\setlength{\LTright}{-9mm plus 1fill}
\setlength{\LTcapwidth}{174mm}
\small
\begin{longtable}{|c|>{\centering\arraybackslash}p{32mm}||c|>{\centering\arraybackslash}p{32mm}||c|>{\centering\arraybackslash}p{32mm}||c|>{\centering\arraybackslash}p{32mm}|}
\caption{The $262$ trivial equidistribution classes of length-two mesh patterns.}
\label{tab:involution-trivial}\\
\hline
\multicolumn{1}{|c|}{\footnotesize Class} &
\multicolumn{1}{c||}{\footnotesize trivial} &
\multicolumn{1}{c|}{\footnotesize Class} &
\multicolumn{1}{c||}{\footnotesize trivial} &
\multicolumn{1}{c|}{\footnotesize Class} &
\multicolumn{1}{c||}{\footnotesize trivial} &
\multicolumn{1}{c|}{\footnotesize Class} &
\multicolumn{1}{c|}{\footnotesize trivial} \\
\hline
\endfirsthead

\multicolumn{8}{c}{\tablename\ \thetable\ (continued)}\\
\hline
\multicolumn{1}{|c|}{\footnotesize Class} &
\multicolumn{1}{c||}{\footnotesize trivial} &
\multicolumn{1}{c|}{\footnotesize Class} &
\multicolumn{1}{c||}{\footnotesize trivial} &
\multicolumn{1}{c|}{\footnotesize Class} &
\multicolumn{1}{c||}{\footnotesize trivial} &
\multicolumn{1}{c|}{\footnotesize Class} &
\multicolumn{1}{c|}{\footnotesize trivial} \\
\hline
\endhead

\hline
\multicolumn{8}{r}{\footnotesize Continued on the next page}\\
\endfoot

\hline
\endlastfoot
{\scriptsize\texttt{INV-001}} & $\pattern{scale=0.48}{2}{1/1,2/2}{}$ & {\scriptsize\texttt{INV-002}} & $\pattern{scale=0.48}{2}{1/1,2/2}{0/0}\hspace{0.5mm}\pattern{scale=0.48}{2}{1/1,2/2}{2/2}$ & {\scriptsize\texttt{INV-003}} & $\pattern{scale=0.48}{2}{1/1,2/2}{1/0}\hspace{0.5mm}\pattern{scale=0.48}{2}{1/1,2/2}{0/1}\hspace{0.5mm}\pattern{scale=0.48}{2}{1/1,2/2}{2/1}\hspace{0.5mm}\pattern{scale=0.48}{2}{1/1,2/2}{1/2}$ & {\scriptsize\texttt{INV-005}} & $\pattern{scale=0.48}{2}{1/1,2/2}{2/0}\hspace{0.5mm}\pattern{scale=0.48}{2}{1/1,2/2}{0/2}$ \\ \hline
{\scriptsize\texttt{INV-006}} & $\pattern{scale=0.48}{2}{1/1,2/2}{0/0,2/0}\hspace{0.5mm}\pattern{scale=0.48}{2}{1/1,2/2}{0/0,0/2}\hspace{0.5mm}\pattern{scale=0.48}{2}{1/1,2/2}{2/0,2/2}\hspace{0.5mm}\pattern{scale=0.48}{2}{1/1,2/2}{0/2,2/2}$ & {\scriptsize\texttt{INV-007}} & $\pattern{scale=0.48}{2}{1/1,2/2}{1/0,2/0}\hspace{0.5mm}\pattern{scale=0.48}{2}{1/1,2/2}{2/0,2/1}\hspace{0.5mm}\pattern{scale=0.48}{2}{1/1,2/2}{0/1,0/2}\hspace{0.5mm}\pattern{scale=0.48}{2}{1/1,2/2}{0/2,1/2}$ & {\scriptsize\texttt{INV-008}} & $\pattern{scale=0.48}{2}{1/1,2/2}{0/0,1/0,2/0}\hspace{0.5mm}\pattern{scale=0.48}{2}{1/1,2/2}{0/0,0/1,0/2}\hspace{0.5mm}\pattern{scale=0.48}{2}{1/1,2/2}{2/0,2/1,2/2}\hspace{0.5mm}\pattern{scale=0.48}{2}{1/1,2/2}{0/2,1/2,2/2}$ & {\scriptsize\texttt{INV-009}} & $\pattern{scale=0.48}{2}{1/1,2/2}{1/0,0/1}\hspace{0.5mm}\pattern{scale=0.48}{2}{1/1,2/2}{2/1,1/2}$ \\
\hline
\end{longtable}
\endgroup
\end{document}